\documentclass[12pt]{amsart}
\usepackage{amsmath, amssymb, mathrsfs, textcomp, colonequals,
  comment,soul}
\usepackage{epic}
\usepackage[shortlabels]{enumitem}
\usepackage{xypic}
\usepackage[backref]{hyperref}
\usepackage[alphabetic,backrefs,lite]{amsrefs}
\usepackage{amsthm}
\usepackage{thmtools}
\usepackage[usenames,dvipsnames]{color}
\usepackage{fullpage}
\usepackage{cleveref}
\usepackage{tikz}

\hypersetup{
  colorlinks   = true, 
  urlcolor     = blue, 
  linkcolor    = blue, 
  citecolor   = blue 
}

\numberwithin{equation}{section}
\allowdisplaybreaks[1]

\newtheorem{theorem}[equation]{Theorem}

\newtheorem{corollary}[equation]{Corollary}
\newtheorem{prop}[equation]{Proposition}
\newtheorem{lemma}[equation]{Lemma}

\theoremstyle{definition}
\newtheorem{assumption}[equation]{Assumption}
\newtheorem{properties}[equation]{Properties}
\newtheorem{remark}[equation]{Remark}
\newtheorem{notation}[equation]{Notation}
\newtheorem{defn}[equation]{Definition}

\newtheorem{example}[equation]{Example}

\setenumerate{itemsep=2pt}

\newcommand{\padma}[1]{{\color{Magenta} \sf $\clubsuit\clubsuit\clubsuit$ Padma: [#1]}}
\newcommand{\andrew}[1]{{\color{Green} \sf $\diamondsuit\diamondsuit\diamondsuit$ Andrew: [#1]}}
\newcommand{\connor}[1]{{\color{Orange} \sf $\diamondsuit\diamondsuit\diamondsuit$ Connor: [#1]}}

\newcommand{\ol}[1]{\overline{#1}}
\newcommand{\mc}[1]{\mathcal{#1}}

\newcommand{\Q}{\mathbb Q}
\newcommand{\rats}{\mathbb Q}
\newcommand{\Z}{\mathbb Z}
\newcommand{\ints}{\mathbb Z}
\newcommand{\N}{\mathbb N}

\renewcommand{\O}{\mathcal{O}}

\renewcommand{\P}{\mathbb P}
\newcommand{\proj}{\mathbb P}

\renewcommand{\phi}{\varphi}

\newcommand{\Spec}{\mathrm{Spec} \ }

\DeclareMathOperator{\len}{length}

\renewcommand{\c}[1]{\overline{#1}}
\newcommand{\F}{\mathcal{F}}

\renewcommand{\l}[2]{\len_{\O_K}(\mathcal{O}_{#1}/\mathcal{O}_{#2})}

\newcommand{\X}{\mathcal{X}}
\newcommand{\Y}{\mathcal{Y}}
\newcommand{\tx}[1]{\text{#1}}

\DeclareMathOperator{\Gal}{Gal}

\DeclareMathOperator{\GL}{GL}
\DeclareMathOperator{\divi}{div}
\DeclareMathOperator{\disc}{disc}
\DeclareMathOperator{\sep}{sep}

\DeclareMathOperator{\cha}{char}
\DeclareMathOperator{\Art}{Art}
\DeclareMathOperator{\Aut}{Aut}

\DeclareMathOperator{\lcm}{lcm}
\DeclareMathOperator{\db}{db}
\DeclareMathOperator{\Proj}{Proj}
\DeclareMathOperator{\rad}{rad}
\DeclareMathOperator{\red}{red}
\DeclareMathOperator{\Frac}{Frac}

\DeclareMathOperator{\cdc}{cdc}
\DeclareMathOperator{\cedc}{cedc}
\DeclareMathOperator{\HJL}{HJL}
\DeclareMathOperator{\chara}{char}
\DeclareMathOperator{\ord}{ord}
\DeclareMathOperator{\mult}{\mu}

\DeclareMathOperator{\Supp}{Supp}

\title[Conductor-discriminant inequality]{Conductor-discriminant
  inequality for tamely ramified cyclic covers I}

\author{Andrew Obus}
\address{Baruch College}
\curraddr{1 Bernard Baruch Way, New York, NY 10010, USA}
\email{andrewobus@gmail.com}
\thanks{The first and third authors were supported by NSF CAREER grant DMS-2047638.}

\author{Padmavathi Srinivasan}
\address{Boston University}
\curraddr{665 Commonwealth Avenue, Boston, MA 02215, USA}
\email{padmask@bu.edu}
\thanks{The second author was supported by National Science Foundation grant DMS-2401547.}

\author{Connor Stewart}
\address{Graduate Center, City University of New York}
\curraddr{365 5th Avenue, New York, NY 10016, USA}
\email{connor.stewart95@login.cuny.edu}


\date{\today}

\begin{document}

\maketitle

\begin{abstract}
We prove conductor-discriminant inequalities for all $\ints/n$-covers of $\proj^1$ defined over discretely valued
fields $K$ with excellent valuation ring $\mc{O}_K$ and perfect residue field of characteristic not dividing
$n$, modulo some calculations appearing in work \cite{Part2} of
the third author. Specifically,
when such a curve $X$ is given by $y^n = f(x)$ with $f(x) \in \mc{O}_K[x]$
and $n\mid\deg(f)$, and if $\mc{X}$ is its minimal regular model over
$\mc{O}_K$, then the negative of the Artin conductor of $\mc{X}$ is bounded
above by $(n-1)v_K(\disc(\rad f))$.  This is a direct generalization of
previous work of the first two authors on hyperelliptic curves,
which in turn generalized work of Ogg, Saito, Liu, and the second author. When $f$ is monic, this strengthens a result of Kohls stating that the conductor exponent of the Jacobian of such a curve is bounded above by $(n-1)v_K(\disc(\rad f))$.
\end{abstract}


\section{Introduction}
We prove conductor-discriminant inequalities for all $\ints/n$-covers
of $\proj^1$ defined over Henselian discretely valued
fields $K$ with excellent\footnote{The assumption that $K$ has excellent valuation ring is mild and holds in many common settings, e.g.,\,whenever $K$ is complete. We use this assumption for Lemmas \ref{lemma: local resolution is finite} and \ref{lem:resolveinorder} and for Proposition \ref{prop: non-negativity of cdc and existence of contractions}.} valuation ring and perfect residue field of characteristic not dividing
$n$. This generalizes a previous result of the first two authors
(\cite{OShyper_published}) covering the case $n = 2$, which in turn
generalized earlier work of Ogg, Saito, and Liu on genus $1$ and $2$
curves (\cite{saito2}, \cite{Li:cd}).

\subsection{Main theorem}\label{Smain}
Let $K$ be a Henselian discretely valued field with perfect residue field
$k$. Let $\mc{O}_K$ be the ring of integers of $K$. Assume that
$\mc{O}_K$ is excellent. Let $\nu_K \colon K \rightarrow \Z \cup \{ \infty
\}$ be the corresponding discrete valuation.  Fix a uniformizer
$\pi_K$ of $K$.  Let $X$ be a nice (i.e. smooth,
projective, geometrically integral) curve of genus $g \geq 1$ defined
over $K$.  Let $\mc{X}$ be a proper,
flat, regular $\mc{O}_K$-scheme with generic fiber $X$. The \emph{Artin conductor} associated to the model $\mc{X}$ is defined by
\[ \Art (\mc{X}/\mc{O}_K) =  \chi(\mc{X}_{\overline{K}}) - \chi(\mc{X}_{\overline{k}}) - \delta(X/K) ,\]
where $\chi$ is the Euler characteristic for the $\ell$-adic cohomology and $\delta(X/K)$ is the Swan conductor associated to the $\ell$-adic representation $\Gal (\overline{K}/K) \rightarrow \Aut_{\Q_\ell}
(H^1_{\mathrm{et}}(\mc{X}_{\overline{K}}, \Q_\ell))$ ($\ell \neq \cha k$). The Artin conductor is a measure of degeneracy of the model $\mc{X}$; it is a non-positive integer that is zero precisely when
$\mc{X}/\mc{O}_K$ is smooth or when $g=1$ and
$(\mc{X}_k)^{\mathrm{red}}$ is smooth. For instance, if $\mc{X}/\mc{O}_K$ is a regular, semistable model, then $-\Art(\mc{X}/\mc{O}_K)$ equals the number of singular points of the special fiber $\mc{X}_k$. 

In \cite[Proposition~1.1]{Bloch}, Bloch showed that the Artin conductor appears in the functional equation for the $\zeta$-function attached to a global proper flat regular $\mc{O}_K$-model $\mc{X}$ for a curve $X$ defined over a number field $K$, analogous to how the conductor exponent of the Jacobian of $X$ appears in the usual $\zeta$-function attached to $X$. Databases of curves such as the LMFDB \cite{LMFDB} order curves by the conductor exponent of their Jacobian, and this requires algorithms to compute or at least bound this invariant, which is a difficult problem in general.  In this paper, we describe an approach involving explicit regular models for $\ints/n$-covers of $\proj^1_K$, generalizing previous work of the first two authors when $n=2$. In some cases, this bound improves on known upper bounds for the conductor exponent computed using explicit semistable models (see \cite{DDMM} and \cite{BW_Glasgow} for explicit constructions of semistable models for hyperelliptic/superelliptic curves, and computation of related invariants, and \cite{Kohls} for bounds on the conductor exponent for superelliptic curves). 

The upper bound is in terms of a more easily computable measure of degeneracy $\Delta_{X/K}$ available to a $\ints/n$-cover $X$ of $\proj^1$ (see \Cref{Ddiscriminant} below). The invariant $\Delta_{X/K}$ detects degeneracy of the branch locus on
reduction from the defining equation of the cover. To define $\Delta_{X/K}$, we begin by fixing an integral Weierstrass equation for $X$, namely an affine equation for $X$ of the form $y^n=f(x)$ with $f(x) \in \mc{O}_K[x]$ that is $n$th power-free.

\begin{defn}\label{Dpolydisc}
Assume $f(x) \in \mc{O}_K[x]$ is $n$th power-free and non-constant. Let $\widetilde{f}_n(x_0, x_1) \colonequals x_0^df(x_1/x_0)$, where
$d = n\lceil \deg(f)/n \rceil$; this is the homogenization of
$f$, with extra $x_0$ factors if $n \nmid \deg(f)$.  Let $g \colonequals \rad(\widetilde{f}_n) \in \mc{O}_K[x]/\mc{O}_K^{\times}$, where the radical of a polynomial is the product of its distinct irreducible factors. The \emph{$n$-adapted discriminant} $\disc_n(f)$ of $f$ is defined to be $\disc(g) \in \mc{O}_K/\mc{O}_K^\times$,
where $\disc(g)$ is the homogeneous discriminant defined in \cite[Ch.\
12, (1.22),(1.29)]{GKZ}.
\end{defn}

Observe that $x_0$ divides $g$ as above precisely when $x=\infty$ is a ramification point of the cover $X \to \proj^1_K$. See \Cref{Lcollisions} for why $\disc_n(f)$ detects collisions between the components of the branch divisor of the $\ints/n$-cover upon reduction to $k$.

\begin{remark}\label{Rhyperelliptic}
If $n = 2$, then $\widetilde{f}_2 = \rad(\widetilde{f}_2) \mod \mc{O}_K^\times$ for $f$ as above, and Definition~\ref{Ddiscriminant} of
$\disc_2(f)$ reduces to the definition $\disc'(f)$ given in \cite{OShyper_published}. 
\end{remark}

\begin{defn}\label{Ddiscriminant} A \emph{minimal Weierstrass equation} is an equation for which the integer $\nu_K(\disc_n(f))$ is as small as possible amongst all integral Weierstrass equations for the curve $X$, and
the \emph{minimal discriminant} $\Delta_{X/K}$ of $X$ is
$\nu_K(\disc_n(f))$ for such an equation. \end{defn}

For a curve $X$ as above, let $\mc{X}$ be the normalization of the standard model
$\proj^1_{\mc{O}_K} \colonequals \Proj \mc{O}_K[x_0, x_1]$ of $\proj^1_K$ corresponding to the coordinate
$x \colonequals x_1/x_0$ in the function field $K(X)$. Then $\mc{X}$ is a smooth model for $X$ if and only if $\Delta_{X/K} = 0$ (see \Cref{Lcollisions}~\ref{Lnormalsmooth}).
Conversely, if $\mc{X}$ is a smooth proper model for a $\ints/n$-cover of $\proj^1$, then $\mc{X}$ arises as the normalization of $\proj^1_{\mc{O}_K}$ in $K(X)$, and $\nu_K(\disc_n(f)) = 0$ for a polynomial $f$ giving rise to an equation for the cover as $y^n=f(x)$ (see \Cref{Pquotsmoothmodel}).



The main theorem of the paper is the following:
\begin{theorem}[Conductor-discriminant inequality]\label{Tcdi}
Let $K$ be the fraction field of an excellent, Henselian\footnote{In fact, the Henselian hypothesis can be removed since the discriminant and Artin conductor are invariant under unramified base change and regular models satisfy \'{e}tale descent.} discrete
valuation ring $\mc{O}_K$
with perfect residue field of characteristic $p$ ($p=0$
is allowed), and let
$f \in \mc{O}_K[x]$ be a $n$th power-free polynomial.
Let $X$ be the cyclic cover of $\proj^1_K$  with affine 
equation $y^n=f(x)$.  Assume $p \nmid n$ and that the genus of $X$ is $\geq 1$. Then there exists a proper flat regular $\mc{O}_K$-model $\mc{X}_f$ of $X$ such that 
\begin{equation}\label{Econddiscineq} -\Art(\mc{X}_f/\mc{O}_K) \leq (n-1)\nu_K(\disc_n(f)). \end{equation}
\end{theorem}
We call (\ref{Econddiscineq}) the \textit{conductor-discriminant inequality for $f$}. 

  If $g(X) \geq 1$, let $\Art (X/K)$ denote the Artin conductor associated to the
\emph{minimal} proper regular model of $X$ over $\mc{O}_K$.  

\begin{corollary}\label{Tfinalthm}
Let $X \to \proj^1_K$ be a $\ints/n$-cover with $g(X) \geq 1$.
 Let $\Delta_{X/K}$ be the minimal discriminant of $X$
 and let $\Art(X/K)$ denote the Artin conductor of the minimal regular
 model of $X$. Then
 $-\Art(X/K) \leq (n-1) \Delta_{X/K}$.
\end{corollary}

\begin{remark} Let $X$ be the curve with affine equation $y^n=x^n-\pi_K$. Then the normalization $\X$ of $\P^1_{\O_K}$ in $K(X)$ is the minimal regular model of $X$, and $-\Art(X/K) = (n-1)^2$ (see \Cref{example:optimal} for a proof). A direct computation shows that $\Delta_{X/K} = n-1 $. Hence $-\Art(X/K) = (n-1) \Delta_{X/K}$, showing that the factor $n-1$ multiplying $\Delta_{X/K}$ in Corollary \ref{Tfinalthm} is optimal for any $n$ prime to $p$.  
    

\end{remark}

In fact, a similar inequality to \ref{Tfinalthm} was obtained by Kohls in his thesis (\cite[Theorem 5.1.11]{Kohls}) with
$-\Art(X/K)$ replaced by the \emph{conductor exponent} $\phi(X/K)$. The invariant $\phi(X/K)$
is the conductor of the 
$\Gal(\overline{K}/K)$-representation $H^1_{\mathrm{et}}(X_{\overline{K}}, \Q_\ell)$  for any $\ell \neq \cha k$. It depends only on $X$ and not on any particular integral model for $X$.
The following result of Liu relates $\phi(X/K)$ to $\Art(\mc{X}/\mc{O}_K)$ for a proper flat regular $\mc{O}_K$-model
$\mc{X}$ of $X$:

\begin{prop}[{\cite[Proposition~1]{Li:cd}}]\label{Partinconductor}
  If $\mc{X}/\mc{O}_K$ is a proper flat regular model of $X$, then
  $$-\Art(\mc{X}/\mc{O}_K) = N - 1 + \phi(X/K),$$
where $N$ is the number of irreducible components of the (geometric)
special fiber $\mc{X}_{s}$ of $\mc{X}$.
\end{prop}

Applying Proposition~\ref{Partinconductor} to the model $\mc{X}_f$ in Theorem~\ref{Tcdi} 
immediately yields the so-called ``conductor exponent-discriminant
inequality'' below, which is Kohl's result when $f$ is monic. 

\begin{corollary}[Conductor exponent-discriminant inequality]\label{Ccedi}
In the situation of Theorem~\ref{Tcdi}, we have $\phi(X/K) \leq (n-1)\nu_K(\disc_n(f))$.
\end{corollary}

In Corollary~\ref{cor: exponent-discriminant ineq} we give an easier
proof of Corollary~\ref{Ccedi} without using Theorem~\ref{Tcdi}.

\subsection{Method of proof}\label{Smethod} Theorem~\ref{Tcdi} is fundamentally about bounding the number of
irreducible components of the geometric special fiber of a regular model of $X$.  The idea, similar
to that of \cite{OShyper_published} in the hyperelliptic case, is to
start by considering the normalization $\nu_0 \colon \mc{X}_0 \to
\proj^1_{\mc{O}_K}  \equalscolon \mc{Y}_0$ in $K(X)$.  The
singularities on $\mc{X}_0$ arise from singularities of the branch locus
of $\nu_0$, and the invariant $\disc_n(f)$ measures the ``overall badness'' of these singularities (see \Cref{Lcollisions} and \Cref{Pquotsmoothmodel}).  In the hyperelliptic case, one successively
blows up reduced closed points on $\proj^1_{\mc{O}_K}$, until the
cover $\nu_n \colon \mc{X}_n \to
 \mc{Y}_n$ obtained by normalizing the blown up model $\mc{Y}_n$ in $K(X)$ has a regular branch divisor. This guarantees that the resulting model $\mc{X}_n$ of $X$ is regular, and
keeping careful track of how much the branch locus singularities are
improved by blowing up at every step suffices to yield the conductor-discriminant inequality.  Essentially, there is a competition between how quickly one resolves the branch locus singularities and how many irreducible components are added to the special fiber of the model of $X$, and the bookkeeping can be arranged such that, at every step, the ``resolving the branch locus'' side wins.

In our situation, where the cover $X \to \proj^1_K$ is cyclic of
degree $n$, applying this method runs into several obstacles.  Most importantly, for $n >
3$, it is in general \emph{not possible} to find a regular model of
$\proj^1_K$ whose normalization in $K(X)$ is regular (see
\cite[Proposition 6.6]{LL} with $g = 0$ in that proposition).  In
particular, it is not always possible to start with $\proj^1_{\mc{O}_K}$ and blow up successive reduced closed
points to obtain a model whose normalization in $K(X)$ has regular
branch locus.  We attempt to solve this problem by blowing up until we
obtain a regular model $\mc{Y}$ of $\proj^1_K$
such that the branch locus of its normalization $\mc{X}$ in $K(X)$ is a \emph{simple normal crossings} divisor, that is,
consists of regular irreducible components that meet transversely.  It
is indeed possible to do this, and in
this case, every singularity of $\mc{X}$ is a \emph{tame cyclic quotient singularity}, which has an explicit resolution in terms of
Hirzebruch--Jung theory.

We play a similar game to that of \cite{OShyper_published}, comparing
the improvement in the the branch locus singularities to the
number of components added to the special fiber of the model of $X$ after every blow up.  In the case
of general $n$, quantifying this improvement is significantly more complicated,
especially when dealing with a branch locus singularity of
multiplicity $2$. These calculations are carried out in detail in 
\cite{Part2} by the third author, and we quote the main results as
Propositions~\ref{prop: non-negativity for cedc} and \ref{prop:
  non-negativity of cdc and existence of contractions}.  In the end, the regular model of $X$ resulting from
resolving all the tame cyclic quotient singularities in $\mc{X}$ usually satisfies Theorem~\ref{Tcdi}, but unfortunately the special fiber sometimes has too many irreducible components.

To resolve this final issue, we do not blow up
$\proj^1_{\mc{O}_K}$ all the way to the point where the branch locus
has normal crossings. Instead we stop in certain cases where the branch locus has a
multiplicity 2 point whose local neighborhood, when normalized in
$K(X)$, yields a \emph{rational double point} (or Du Val
singularity).  It turns out that applying the well-known resolutions of these
singularities from \cite{Lip69} is more efficient than the
process of resolving the branch locus to normal crossings and then
resolving the resulting tame cyclic quotient singularities, and this is good
enough to prove Theorem~\ref{Tcdi}.

\subsection{Other results}\label{Sindependent}

Along the way to our proof of Theorem~\ref{Tcdi}, we prove two other
results that may be of independent interest.

\begin{itemize}
  \item If $X \to \proj^1_K$ is an $\ints/n$-cover, we give a formula (see \Cref{Pswanconductor}) for the Swan
    conductor of $X$ in terms of its
    defining equation when $n$ is coprime to the residue
    characteristic.  This formula is similar to one of Spencer
    (\cite[Corollary~1.7, Remark~4.3]{Spencer}) for the Swan conductor
     of a curve $X$ equipped with an \emph{arbitrary} degree $n$ map
     to $\proj^1_K$, although Spencer has the slightly
    stronger assumption that $\chara k > n$. Our
    technique is different from that of \cite{Spencer}, and yields a
    significantly shorter proof in the cyclic case.
  \item In Proposition~\ref{prop: resolution of tame cyclic quotient
      singularities}, we explicitly resolve tame cyclic quotient singularities on
    arithmetic surfaces in a
    more general context than that given in, say, \cite{HN}, where the
    singularity is assumed to arise as a tame cyclic quotient of a local
    arithmetic surface that is regular
    \emph{with normal crossings}.  We prove that the standard
    Hirzebruch--Jung-inspired algorithm works without the assumption
    of normal crossings, and that it in fact gives the
    \emph{minimal} resolution. 
\end{itemize}

\subsection{Outline of paper}\label{Soutline}
In \S\ref{Sdiscgeometry}, we explain why the $n$-adapted discriminant detects collisions between the branch points of the $\ints/n$-cover on reduction. In \S\ref{Salgclosedreduction}, we explain why we may assume that the residue field $k$ is algebraically closed throughout without any loss of generality. In \S\ref{Sswan}, we
derive a formula for the Swan conductor. In \S\ref{Sdb}, we introduce the
\emph{discriminant bonus}, a quantity that is useful when we rephrase
our inequalities. In \S\ref{Sfurtherredn}, we further reduce to the case where the degree of the defining polynomial $f$ in the defining equation $y^n=f(x)$ for the cover is divisible by $n$. In \S\ref{Stame}, we adapt known results about explicit
resolution of tame cyclic quotient singularities to our situation. In \S\ref{SEulerchar} we derive an inclusion--exclusion formula for the Artin conductor using properties of the compactly supported \'{e}tale Euler characteristic function. In \S\ref{Srephrasing}, we
rewrite the conductor-discriminant (resp.\ conductor exponent-discriminant)
inequality for $X$ in terms of the positivity of a quantity called the
``conductor discriminant contribution'' (resp.\ ``conductor
exponent-discriminant contribution'') associated to multiplicity 2 isolated
singularities on particular models of $X$.  This quantity encodes the
``resolution versus components'' game alluded to in \S\ref{Smethod}.
Lastly, in
\S\ref{Scontribution}, we use results from \cite{Part2} to prove the
rephrased inequalities from \S\ref{Srephrasing}.


\section*{Notation and conventions}
Throughout, $K$ is a Henselian field with respect to a
discrete valuation $\nu_K$, and $\mc{O}_K$ is its valuation ring,
assumed to be excellent, and $k$ is its residue field, assumed to be perfect.  
We denote separable and algebraic
closures of $K$ by $K^{\sep} \subseteq \ol{K}$.
We fix a uniformizer $\pi_K$ of $\nu_K$
and normalize $\nu_K$ so that $\nu_K(\pi_K) = 1$.

For a finite separable field extension $L/K$, we let $\disc(L/K)$
denote the discriminant of the field extension $L/K$ and let
$\Delta_{L/K} \colonequals \nu_K(\disc(L/K))$. 


For an integral $K$-scheme or $\mc{O}_K$-scheme $S$, we denote the corresponding function
field by $K(S)$. If $\mc{Y} \rightarrow \mc{O}_K$ is an arithmetic
surface, an irreducible codimension 1 subscheme of $\mc{Y}$ is called
\emph{vertical} if it lies in a fiber of $\mc{Y} \to \mc{O}_K$, and
\emph{horizontal} otherwise.  For two divisors $D,D'$ on a normal arithmetic surface $\mc{Y}$, we say $D \leq D'$ or $D' \geq D$ if $D'-D$ is an effective Weil divisor on $\mc{Y}$. Let $f \in K(\mathcal{Y})$. We denote the
divisor of zeroes of $f$ by $\divi_0(f)$. 
If $E$ is a regular codimension $1$ point of $\mc{Y}$, we will denote
the corresponding discrete valuation on the function field of $\mc{Y}$
by $\nu_E$. If $P$ is a closed point on $\mc{Y}$, we denote the
corresponding local ring by $\mc{O}_{\mc{Y},P}$ and maximal ideal by
$\mathfrak{m}_{\mc{Y},P}$. By a singular point of an arithmetic
surface $\mc{Y}$, we mean a closed point $P$ such that $\len
\mathfrak{m}_{\mc{Y},P}/\mathfrak{m}_{\mc{Y},P}^2 \neq 2$. For a
Cartier divisor $\Delta$ on an arithmetic surface $\mc{Y}$ and a
closed point $y$ on $\Delta$, we let $\mult_y(\Delta)$ denote the multiplicity of the point $y$ on $\Delta$, i.e., if $f$ is a defining equation for $\Delta$ in $\mc{O}_{\mc{Y},P}$, then $\mult_y(\Delta) = \max \{r \ | \  f \in \mathfrak{m}_{\mc{Y},P}^r \}$. The branch divisor of a finite morphism of arithmetic surfaces is always considered as a reduced scheme.

Throughout this paper, we will let $X \to \proj^1_K$ be a
$\ints/n$-cover of $\proj^1$ over $K$ of genus $g \geq 1$.  Write $\P^1_{K} =
\Proj K[x_0,x_1]$ and let $x \colonequals x_1/x_0$. Let $\P^1_{\O_K}
\colonequals \Proj \mc{O}_K[x_0,x_1]$. We assume that $X$ is given by an
(affine) equation $y^n = f$, with $f \in \mc{O}_K[x]$ that is $n$-th power free.  

On a regular arithmetic surface, the intersection number of two
vertical divisors $D_1$ and $D_2$ will be written $(D_1, D_2)$.

\section*{Acknowledgments}
In an earlier version of the paper, we had originally thought that resolutions of tame cyclic quotient singularities in the context of Proposition \ref{prop: resolution of tame cyclic quotient singularities} were not necessarily minimal. In fact, they are always minimal, and we thank Kioshi Morosin for help identifying an error in a purportedly non-minimal example. 

\section*{AI Statement}
This work was conducted without the use of AI. In addition, the main results and their proofs were known to us by January 2025 and discussed by the authors in public talks in 2025 and early 2026, prior to the recent surge of AI-assisted proofs.

\section{The geometry of the $n$-adapted discriminant}\label{Sdiscgeometry}

In this section, we prove that the $n$-adapted discriminant $\disc_n(f)$ from Definition~\ref{Dpolydisc} detects collisions between the
components of the branch divisor of the cover $y^n=f(x)$ (\Cref{Lcollisions}), and thus serves as a measure of degeneracy for a $\ints/n$-cover of $\proj^1_K$ (\Cref{Pquotsmoothmodel}). Recall $\proj^1_{\mc{O}_K} \colonequals \Proj \mc{O}_K[x_0, x_1]$ is the standard model $\proj^1_K$, that   
$x \colonequals x_1/x_0$, and that $X$ is the curve with affine equation $y^n=f(x)$. Let $\mc{X}$ be the normalization of  $\proj^1_{\mc{O}_K}$
 in $K(X)$.  Let $B \subseteq \proj^1_{\mc{O}_K}$ be the branch locus of
$\mc{X} \to \proj^1_{\mc{O}_K}$ with its reduced induced subscheme structure. Write $B = H + V$ for the decomposition of $B$ into
its horizontal and vertical parts.  If $\phi$ is the structure
morphism $H \to \mc{O}_K$, let $\disc(H \to \mc{O}_K)$ be the closed
subscheme $\Delta(\phi)$ of $\Spec
\mc{O}_K$ defined in \cite[Definition V-21]{EH}. 

\begin{lemma}\label{Lcollisions}
\hfill
 \begin{enumerate}[\upshape (i)]
 \item $$\nu_K(\disc_n(f)) = \begin{cases} \len (\disc(H \to
 \mc{O}_K)) \textup{ \quad \quad \quad \quad \quad \quad \quad if $V=0$, and,} \\
 \len (\disc(H \to
 \mc{O}_K)) + 2H \cdot V - 2 \textup{ \quad if $V \neq 0$.} \end{cases}$$ 
  In particular $\nu_K(\disc_n(f)) = 0$ if and only if the map $B \rightarrow \mc{O}_K$ is smooth.
 \item\label{Lnormalsmooth}  The model $\mc{X}$ is a
   smooth model for $X$ if and only if $\nu_K({\disc_n(f)}) = 0$. 
\end{enumerate}
\end{lemma}

\begin{proof}
  \hfill
 \begin{enumerate}[\upshape (i)] 
  \item Recall from \Cref{Dpolydisc} that we defined $\widetilde{f}_n(x_0, x_1) \colonequals x_0^df(x_1/x_0)$, where
$d = n\lceil \deg(f)/n \rceil$ and $g = \rad(f_n) \in \mc{O}_K[x_0,x_1]/\mc{O}_K^\times$. By construction $\divi(g)=B$, and further if we write $g = ch$ with $c
\in \mc{O}_K$ and $h \in \mc{O}_K[x_0,x_1]$ such that $\pi_K$ does not divide $h$, then $V(c) =
V$ and $V(h) = H$. By \Cref{Dpolydisc} and \cite[Chapter~12, Equation~(1.23)]{GKZ}, 
\[ \disc_n(f) = \disc(g) = c^{2(\deg(h) - 1)}\disc(h),\]
which implies that
\[ \nu_K(\disc_n(f)) = \nu_K(\disc(h)) + (2\deg(h)-2) \nu_K(c).\]
Since $\nu_K(c) = 1$ if $V \neq 0$ and $\nu_K(c) = 0$ if $V = 0$, and since $\deg(h) = \deg(h \mod \pi_K)$ by our assumption that $\pi_K$ does not divide $h$, and since $\deg(h \mod \pi_K) = H \cdot V$ by definition, it follows that $(2\deg(h)-2) \nu_K(c)$ equals $2H \cdot V-2$ if $V \neq 0$ and is $0$ if $V=0$.
Since $\disc(h)$ cuts out the subscheme $\disc(H \to \mc{O}_K)$ of $\Spec
\mc{O}_K$ by \cite[Corollary V-23]{EH}, it follows that
$\nu_K(\disc(h)) = \text{length}(\disc(H \to \mc{O}_K))$. Combining the last two sentences we get the desired equality.

Since $X$ has genus $\geq 1$, the cover $X \to \proj^1_K$ is branched at at least three points, so when $V \neq 0$ we have $H \cdot V \geq 3$. Combined with the previous paragraph, it follows that $\nu_K({\disc_n(f)}) = 0$
precisely when the components of the branch divisor $B$ are pairwise
disjoint and $H \rightarrow \mc{O}_K$ is \'{e}tale. These
conditions are together equivalent to $B \rightarrow \mc{O}_K$ being
\'{e}tale.

\item In light of part (i), we prove that $\mc{X} \to \mc{O}_K$ is
  smooth if and only if $B \to \mc{O}_K$ is smooth.  If $V \neq 0$, then the special
  fiber of $\mc{X}$ is not reduced and thus not smooth.  So we may
  assume $V = 0$ (meaning that $\rad(f)$ is monic), and show
  that the map $\mc{X} \to \mc{O}_K$ is smooth if and only if the map $H \to
  \mc{O}_K$ is smooth.

  For the ``if'' direction, since $k$ is algebraically closed by assumption, it follows that the map $H \to \mc{O}_K$ is
  smooth precisely when $H$ is a disjoint union of
  sections of $\proj^1_{\mc{O}_K} \to \mc{O}_K$, or equivalently that
  $\rad(f)$ splits into linear factors over $K$ and remains separable when
  reduced to $k$. In this case $\mc{X} \to \mc{O}_K$
  is smooth by \cite[Expos\'{e}~XIII, Proposition~5.4]{SGA1}
applied to $S = \Spec \mc{O}_K$ and $X$ being the localization of $\proj^1_{\mc{O}_K}$ at each closed
point of $H$.

  For the ``only if'' direction, assume that
  $\mc{X} \to \mc{O}_K$ is smooth, let $t \in
  \proj^1_{\mc{O}_K}$ be a closed point of $H$, and
  let $z \in \mc{X}$ be a point above $t$.  Then $\hat{\mc{O}}_{\proj^1_{\mc{O}_K},
    t} \cong (\hat{\mc{O}}_{\mc{X}, z})^{I_z}$, where $I_z \leq \ints/n$ is the inertia
  group at $z$, and both rings are power series rings in one variable
  over $\mc{O}_K$. Let $d_k$ be the degree of the ramification divisor
  of the special fiber of $\mc{X} \to \proj^1_{\mc{O}_K}$ at $z$, and
  let $d_K$
  be the sum of degrees of the ramification divisor of the generic fiber $X \to \proj^1_{K}$ at all points
  specializing to $z$. Then $d_k = d_K$ by \cite[I, Claim 3.2]{GM}. A
  geometric point of $\proj^1_K$ with ramification index $e$ for the
  cover contributes $|I_z|(1 - 1/e) \geq |I_z|/2$ to the sum $d_K$.
  We also have that $d_k = |I_z| - 1$. Combining the last three sentences, we conclude that there
  is only one such geometric point, which means that $H \to \mc{O}_K$
  is locally of degree 1 at $t$, or equivalently a local isomorphism, and thus $H \to \mc{O}_K$ is smooth at $t$. \qedhere   
 \end{enumerate}
  \end{proof}

\begin{remark}\label{R:geodisc} Since $V = \mathbb{P}^1_k$ when $V \neq 0$, using the convention that $p_a(\emptyset) = 1$\footnote{This convention is reasonable if we think of $\dim(\emptyset) = -1$ and
   $\chi(\emptyset, \mc{O}_{\emptyset}) = 0$.}, \Cref{Lcollisions} above can be restated more uniformly as
   \begin{equation}\label{Ediscgen} \nu_K(\disc_n(f)) = \len (\disc(H \to
 \mc{O}_K)) + 2H \cdot V + 2p_a(V) - 2. \end{equation}

Thus, we can break $\nu_K(\disc_n(f))$ into three parts: The first part $\disc(H
\to \mc{O}_K)$, accounts for collisions between the
 horizontal components of $B$, the second part
 $2H \cdot V$, accounts for collisions (if any) between
 the unique vertical component of $B$ with the
 horizontal components of $B$, and the third part accounts for collisions between the vertical components of $B$.
In forthcoming work \cite{OSRelative}, the first two authors use \Cref{Ediscgen} to define the discriminant of the branch divisor on an arbitrary regular arithmetic surface $\mc{Y}$ in place of $\proj^1_{\mc{O}_K}$ to study conductor--discriminant inequalities in a more general setting. In the case of general $\mc{Y}$, the vertical branch divisor $V$ may have more than one irreducible component  and the appropriate generalization of $\nu_K(\disc_n(f))$ would need to account for collisions between different vertical branch components.
\end{remark}

\begin{prop}\label{Pquotsmoothmodel}
Let $X$ be a nice curve with a smooth proper $\mc{O}_K$-model $\mc{X}$. Assume that $X$ admits a $\ints/n$-action such that the quotient is isomorphic to $\proj^1_K$. Assume that the genus $g$ of $X$ is at least $1$. Then the $\ints/n$-action extends to $\mc{X}$, and there is an isomorphism $\mc{X}/(\ints/n) \cong \proj^1_{\mc{O}_K}$. Moreover, there is an affine equation for $X$ of the form $y^n=f(x)$ with $f(x) \in \mc{O}_K[x]$ that is $n$th power free and satisfies $\nu_K(\disc_n(f)) = 0$. 
\end{prop}
\begin{proof}
Since $g \geq 1$ by assumption, the minimal proper regular model of
$X$ is unique up to unique isomorphism (cf. \cite[Definition~3.31,
Proposition~3.36~(2)]{Liubook}) and since the smooth model $\mc{X}$ is
the minimal model, it follows that the $\ints/n$-action extends to
$\mc{X}$.  If $\mc{Y} \colonequals \mc{X}/(\ints/n)$, then \cite[Proposition~10.3.48(a)]{Liubook} implies that
$\mc{Y}$ is
smooth, and thus isomorphic to $\proj^1_{\mc{O}_K}$. 

Since $n \nmid \cha(k)$ and since $K$ is Henselian, it follows that $K$ has all $n$th roots of unity. The final claim follows from \Cref{Lcollisions}~\ref{Lnormalsmooth} by choosing explicit coordinates on $\proj^1_{\mc{O}_K}$ and applying Kummer theory. 
\end{proof}

\section{Reduction to the case of algebraically closed residue field}\label{Salgclosedreduction}
In this section, we explain why for Theorem~\ref{Tcdi} and Corollary~\ref{Ccedi}, we may assume that the residue field $k$ is \emph{algebraically
    closed} rather than just perfect for the rest of the paper without any loss of generality. Firstly,
 replacing $K$ with an unramified extension does not affect $\nu_K(\disc_n(f))$.  Secondly, base changes to unramified extensions of $K$ do not affect the conductor exponent $\phi(X/K)$.  Thirdly, (normalized) base changes of any regular model $\mc{X}/\mc{O}_K$ to 
  unramified extensions of $\mc{O}_K$ preserve regularity and
  commute with reduction to the special fiber.  In particular, to
  construct a regular model $\mc{X}_f$ as in Theorem~\ref{Tcdi}, it
  suffices to construct such a model $\mc{X}_f^{ur}$ over the ring of
  integers $\mc{O}_K^{ur}$ of the
  maximal unramified extension $K^{ur}$ of $K$, then let $\mc{X}_f \colonequals \mc{X}_f^{ur}/\Gal(K^{ur}/K)$; observe that the quantities on
  the right hand side of the equation in
  Proposition~\ref{Partinconductor} are the same before and after
  taking the quotient.   Thus $\Art(\mc{X}_f/\mc{O}_K) = \Art(\mc{X}_f^{ur}/\mc{O}_K^{ur})$.  So
  the inequality in Theorem~\ref{Tcdi} holds for $\mc{X}_f$ so long as
  it holds for $\mc{X}_f^{ur}$.  So from now on we assume that $k$ is algebraically closed.

Since $K$ is Henselian and $k$ is algebraically closed with $\cha k \nmid n$, all units in $\mc{O}_K$ are $n$th powers, and the group $K^{\times}/(K^{\times})^n$ has
order $n$ and is
generated by the image of $\pi_K$.  Since multiplying $f$ by $n$th powers does not change the isomorphism class of $X$, without loss of generality we may assume that the maximum power of $\pi_K$ dividing $f$ in $\O_K[x]$ is $\pi_K^b$ for some $b
\in \{0, \ldots, n-1\}$.  

\begin{assumption}\label{Affactorization} Assume $\chara k \nmid n$ and $k$ is algebraically closed. We
write the irreducible factorization of $f$ as 
\begin{equation}\label{Effactorization}
  \pi_K^bf_1^{d_1} \ldots f_r^{d_r},
\end{equation}
where $b \in \{0,
\ldots, n-1\}$, the $d_i \in \{1, \ldots, n-1\}$, and the $f_i \in
\mc{O}_K[x]$ are distinct irreducible separable polynomials. For each $1 \leq i \leq
r$, we pick a root $\alpha_i$ of $f_i$ in $\ol{K}$ and set $K_i = K(\alpha_i)$.
\end{assumption}

\section{A computable formula for the Swan conductor}\label{Sswan}



\begin{prop}\label{Pswanconductor}

Let $f,K_i$ be as in \Cref{Affactorization}, and let $X/K$ be the smooth projective curve with affine equation $y^n =
f(x)$. 
The Swan conductor $\delta(X/K)$ for the curve $X$ equals
\[ \sum_{i=1}^r (n - \gcd(n, d_i)) (\Delta_{K_i/K} - \deg(f_i) + 1) .\]
\end{prop}

\begin{proof}
Let $L/K$ be a finite Galois extension over which $f$ splits, and let
$G_j$ be the $j$th higher ramification group of $G \colonequals \Gal(L/K)$ for the lower
numbering.  Let $B_i$ be the set of roots of $f_i$, and let $B = \bigsqcup_i B_i$.  Since the $f_i$ are irreducible, $G$ acts
on each $B_i$.  Denote the special fiber of
the quasi-stable model $\mathcal{X}^{qs}$ of $X$ constructed in
\cite[\S4]{BW_Glasgow} by $\overline{X}^{qs}$.  The model
$\mathcal{X}^{qs}$ is a semi-stable
model for which all ramification points of $X \to \proj^1$ specialize
to distinct smooth points. 
  By construction of $\mathcal{X}^{qs}$, it follows that $\Gamma \colonequals \Gal(X/\proj^1)$ acts on $\ol{X}^{qs}$ and we write
$\ol{Y}^{qs} \colonequals \ol{X}^{qs}/\Gamma$.  The group $G$ acts on
$\overline{X}^{qs}$ by $k$-automorphisms and commutes with the action
of $\Gamma$, so its action descends to $\ol{Y}^{qs}$.
Since arithmetic genus
is constant in flat families, we have $p_a(\ol{X}^{qs}) = g(X)$ and
$p_a(\ol{Y}^{qs}) = 0$.

By \cite[Theorem 2.13]{BW_Glasgow}, we have the formula
\begin{equation}\label{Ebasicswan}
\delta(X/K) = \sum_{j=1}^{\infty} \frac{|G_j|}{|G_0|} \cdot
(2p_a(\overline{X}^{qs})-2p_a(\overline{X}^{qs}/G_j)). 
\end{equation}

For $j \geq 1$, we claim that
\begin{equation}\label{EGjpart}
2p_a(\overline{X}^{qs}) -
2p_a(\overline{X}^{qs}/G_j) = \sum_{i =1}^r (n- \gcd(n, d_i))(|B_i| -
|B_i/G_j|).
\end{equation}

To prove (\ref{EGjpart}), note that since the actions of $\Gamma$ and $G$
commute, the action of $\Gamma$ descends to $\ol{Y}/G_j$
for any $j$.  Since the group $|G_j|$ is a $p$-group for all $j \geq 1$,
it cannot permute a $\Gamma$-orbit of
$\ol{X}^{qs}$ non-trivially.  Consequently, $\ol{X}^{qs}/G_j \to
\ol{Y}^{qs}/G_j$ is a map of semistable curves
that is generically separable of degree $n$ above each irreducible component of the target. 
By construction, the branch points of $X \to \proj^1$ specialize to
distinct smooth points on $\ol{Y}^{qs}$.  
Identifying $B$ with the branch locus of $X \to \proj^1$ away from $\infty$, we can also
identify the smooth
points of $\ol{Y}^{qs}$ in the branch locus of $\ol{X}^{qs} \to
\ol{Y}^{qs}$ away from the specialization of $\infty$ with $B$ via specialization.  Carrying on, the smooth points of $\ol{Y}^{qs}/G_j$
in the branch locus of $\ol{X}^{qs}/G_j \to \ol{Y}^{qs}/G_j$ away from
the specialization of $\infty$
correspond to the elements of 
$\bigsqcup_i B_i/G_j$.  Furthermore, the ramification index above each
point of $B_i/G_j$ is $n/\gcd(n, d_i)$.  Since $p_a(\ol{Y}^{qs}/G_j) = 0$, the
Riemann-Hurwitz formula for semistable curves (see, e.g.,
\cite[Proposition~2.15]{ObusThesis}) shows
that
\begin{equation}\label{EfullX}
2p_a(\ol{X}^{qs}) - 2 = -2n + 
\sum_{i=1}^r (n - \gcd(n, d_i)) |B_i| + |R_{\infty}|
\end{equation}
and
\begin{equation}\label{EXmodGj}
2p_a(\ol{X}^{qs}/G_j) - 2 = -2n + 
\sum_{i=1}^r (n - \gcd(n, d_i)) |B_i/G_j| + |R_{\infty}|, 
\end{equation}
where $R_{\infty}$ is the ramification divisor above $\infty$.
Combining (\ref{EfullX}) and (\ref{EXmodGj}) proves the claim.

From (\ref{Ebasicswan}) and (\ref{EGjpart}), we obtain
\begin{equation}\label{Efullswan}
\begin{aligned}
    \delta(X/K) &= \sum_{j=1}^{\infty} \frac{|G_j|}{|G_0|} \sum_{i =1}^r (n- \gcd(n, d_i))(|B_i| -
    |B_i/G_j|) \\
    &= \sum_{i=1}^r (n - \gcd(n, d_i)) \sum_{j=1}^{\infty}\frac{|G_j|}{|G_0|}  (|B_i| - |B_i/G_j|).
\end{aligned}
\end{equation}
If $V_i$ is a vector space with $G_j$-stable basis $B_i$, then $\dim(V_i^{G_j}) =
|B_i/G_j|$.  So for each $i$, the sum over $j$ on the right-hand side of
(\ref{Efullswan}) is, by definition, the Swan conductor of the permutation
representation $\rho$ of $G$ on the roots of $f_i$.  Now, $\rho$ is induced from the
trivial representation of $H_i \colonequals \Gal(L/K_i)$.  By a variant of the
conductor-discriminant formula (see \cite[p.\ 101, Corollary to
Proposition 4]{SerreLF}), the Artin conductor of $\rho$
is $\Delta_{K_i/K}$. Since $K_i/K$ is totally ramified, we have $G_0 =
G$ and $V_i^{G_0}$ is generated by $(1, 1, \ldots, 1)$.
$G_0$ acts on $B_i$ with one fixed point.  Thus $\dim(V_i) - \dim(V_i^{G_0}) =
\deg(f_i) - 1$, so the Swan
conductor of $\rho$ is $\Delta_{K_i/K} - \deg(f_i) + 1$.  Plugging
this into (\ref{Efullswan}) proves the proposition.
\end{proof}

\section{The discriminant  bonus}\label{Sdb}

Let $\disc_n(f)$ be the $n$-adapted discriminant of $f$ as in \Cref{Dpolydisc}, and let $K_i$ be as in \Cref{Affactorization}.
\begin{defn}\label{Ddiscbonus}
The \emph{discriminant bonus of $f$ over $K$}, written $\db_K(f)$, is the
quantity $\nu_K(\disc_n(f)) - \sum_{i=1}^r \Delta_{K_i/K}$.
\end{defn}

\begin{remark}
 This generalizes the definition of the discriminant bonus given in \cite[Definition~2.4]{OShyper_published} from $n=2$ to arbitrary $n$.
\end{remark}

\begin{remark}\label{Rcolength}
If $f \in \mc{O}_K[x]$ is irreducible and monic, $\alpha$ is a
root of $f$, and $L = K(\alpha)$, then $\db_K(f)$ is
twice the length of $\mc{O}_L / \mc{O}_K[\alpha]$ as an
$\mc{O}_K$-module.  We do not use this, so we omit the proof.
\end{remark}

\begin{remark}\label{Rdiscprop}
 When $n \mid \deg(f)$, the product formula for the discriminant (cf. \cite[Chapter~12, Equation~(1.23)]{GKZ}) implies that
 \[\nu_K( \disc_n(f)) = \nu_K(\disc(\rad(f)) = \nu_K(\disc(\rad(f/\pi_K^b)))+\begin{cases}0&b=0\\2\deg(\rad(f))-2&b\geq1 \end{cases},\]
 and hence by Definition~\ref{Dpolydisc}, it follows that
 \[ \db_K(f) = \db_K(f/\pi_K^b)+\begin{cases}0&b=0\\2\deg(\rad(f))-2&b\geq1 \end{cases}.\] 
\end{remark}

\section{Reduction to the case where $\deg f$ is divisible by $n$}\label{Sfurtherredn}
Changing coordinates on
$\P^1_K$ using an element of $\GL_2(\mc{O}_K)$ does not change
$\nu_K(\disc_n(f))$ (see, e.g., \cite[Ch.\ 12, (1.26)]{GKZ}).
Since $k$ is
algebraically closed, we may assume by
such a change of coordinates that no component of the branch divisor of $X \rightarrow \proj^1_K$ specializes to $\infty$ (see, e.g., \cite[Section~1.3]{PadmaRational}). In other words, we may assume that all roots of $f$ are integral over $\mc{O}_K$,
and that $n \mid \deg(f)$.   
Thus, for the remainder of the paper, we make the following assumption:

\begin{assumption}\label{Afform} Assume $f,f_i,n$ are as in \Cref{Affactorization}. Additionally assume that the $f_i$ are monic, and that $n \mid \deg(f)$.
\end{assumption}

The argument above proves the following proposition.
\begin{prop}\label{Preduce2}
If the conductor-discriminant inequality holds for all $f$ satisfying Assumption~\ref{Afform}, then it holds for all $f \in \mc{O}_K[x]$.
\end{prop}

\section{Resolution of tame cyclic quotient singularities}\label{Stame}

In this section, we generalize several results about the resolution of tame cyclic quotient singularities from \cite{HN} and \cite{Steck}, dropping the additional ``stable-setting hypothesis'' (\Cref{dfn: tame cyclic quotient admitting stable setting}) made in \cite{HN} and \cite{Steck}.
\begin{defn}\label{dfn: tame cyclic quotient admitting stable setting}
    Let $\nu,r\in\Z$ such that $1\leq r<\nu$, $\chara(k)\nmid \nu$,
    and $\gcd(\nu,r)=1$. Fix a primitive $\nu$-th root of unity $\zeta\in
    \O_K$ and a generator $\sigma\in\Z/\nu\Z$. Consider the
    $\Z/\nu\Z$-action on $\O_K[[t_1,t_2]]$ which is trivial on $\O_K$
    and satisfies $\sigma(t_1)=\zeta t_1$ and
    $\sigma(t_2)=\zeta^rt_2$. Let $\phi\in \O_K[[t_1,t_2]]$ be a power
    series fixed under this $\Z/\nu\Z$-action, and consider the
    induced $\Z/\nu\Z$-action on
    $A\colonequals\O_K[[t_1,t_2]]/(\pi_K-\phi)$.

    Assume $A$ is flat over $\mc{O}_K$, that is, $\pi_K \nmid \phi$.
    An $\O_K$-scheme $\mc{X}$ isomorphic to $\Spec A/({\Z/\nu})$ is called a \textit{tame cyclic quotient singularity of type $\frac{1}{\nu}(1,r)$}. If moreover we can take $A=\O_K[[t_1,t_2]]/(\pi_K-t_1^{a_1}t_2^{a_2})$ for some nonnegative integers $a_1,a_2$ with $a_1\geq1$, then $\mc{X}$ is said to \textit{admit a stable setting} (see \cite[Definition 4.1.5, Theorem 4.1.6, and Theorem 4.3.4]{Steck}).
\end{defn}

\begin{defn}\label{dfn: Hirzebruch--Jung length of pair}
    Let $\nu,r\in\Z$ be as in Definition \ref{dfn: tame cyclic quotient admitting stable setting}. Let $r_0=\nu,r_1=r,$ and while $r_i\neq 0$, let $r_{i+1}$ be the least residue of $-r_{i-1}$ mod $r_i$. The largest index $i$ for which $r_i$ is defined and nonzero is called the \textit{Hirzebruch--Jung length of the pair $(\nu,r)$} and is denoted $L(\nu,r)$.
\end{defn}

\begin{remark}\label{rmk: bound on HJ length of pair} Observe that $L(\nu,r)\leq \nu-1$ in Definition~\ref{dfn: Hirzebruch--Jung length of pair} since the $r_i$ form a decreasing sequence with $r_1=r<\nu$.
\end{remark}

Setting $b_i=\lceil\frac{r_{i-1}}{r_i}\rceil$, we recover the Hirzebruch--Jung (negative) continued fraction expansion for the fraction $\nu/r$ as:
    \[b_1-\frac{1}{b_2-\frac{1}{...-\frac{1}{b_{L(\nu,r)}}}}.\]  
We write $\nu/r = [b_1, b_2, \ldots, b_L]_{HJ}$ with $L = L(\nu, r)$, see
\cite[Definition~4.4.2.4]{HN}.

\begin{prop}\label{prop: resolution of tame cyclic quotient singularities}
    Let $\mc{X}$ be a tame cyclic quotient singularity of type
    $\frac{1}{\nu}(1,r)$ with closed point $x$, let $L=L(\nu,r)$, and
    suppose $\nu/r = [b_1, \ldots, b_L]_{HJ}$.

    \begin{enumerate}[\upshape (i)]
      \item
   The reduced exceptional divisor of the minimal resolution
$\pi:\mc{X}'\to\mc{X}$ is a chain $E_1-E_2-...-E_L$ of $L$ components
isomorphic to $\P^1_k$, i.e., for each $2 \leq i\leq L-1$, the
component $E_i$ intersects $E_{i-1}$ and $E_{i+1}$ each transversely
in a unique point, and
\item For $1 \leq i \leq L$, the self-intersection number of $E_i$ is $-b_i \leq -2$.
  \end{enumerate}
   
\end{prop}

\begin{proof}
   If $\mc{X}$ admits a stable setting, then this is
\cite[Proposition~4.4.2.5]{HN}\footnote{In \cite{HN}, it is assumed
  that $\cha{k}$ does not divide $a_1$ or $a_2$ (called $m_1$ and
  $m_2$ there), but this is only used to define the number $r$, and is
  not used anywhere else in the proof.  Since we define $r$ explicitly
  in our proposition, this assumption is unnecessary.}.
We now prove the result in the
non-stable setting case.
The proof is still
parallel to that of \cite[Proposition~4.4.2.5]{HN} and we adopt its
notation (our $\nu$ plays the role of $n$ in \cite{HN}).  Let $\phi$ be as in
 Definition~\ref{dfn: tame cyclic quotient admitting stable setting}.  The proof of
\cite[Lemma~4.4.2.2]{HN} goes through exactly, replacing
$wt_1^{m_1}t_2^{m_2}$ in \cite{HN} with our $\phi$. As in the proof of \cite[Proposition~4.4.2.5]{HN}, set $u_0 = t_1^{\nu}$, $v_0 =
t_1^{-r}t_2$, and define inductively $u_i = v_{i-1}^{-1}$ and $v_i =
u_{i-1}v_{i-1}^{b_i}$ for $i \geq 1$. 
Letting $S$ be the set of monomials in $t_1$ and $t_2$ invariant under
the $\ints/\nu \ints$-action (\cite[Lemma~4.4.2.2]{HN}), one defines 
$$B_{\sigma_i} = \mc{O}_K[[S]][u_i, v_i]/(\pi_K - \phi_i)
\subseteq \mc{O}_K[[u_i, v_i]]/(\pi_K - \phi_i)$$ as
in the proof of \cite[Proposition~4.4.2.5]{HN}, 
where $\phi_i$ is just $\phi$ written in terms of the variables
$u_i$ and $v_i$.  Note that each monomial of the form
$wt_1^{m_1}t_2^{m_2}$ in $\phi$ (with $w \in \mc{O}_K$) can be written as
$wu_i^{\mu_{i+1}}v_i^{\mu_i}$ as part of $\phi_i$, where the $\mu_i$ are integers defined as in the
proof of \cite[Proposition~4.4.2.5]{HN} applied to the case $\phi =
wt_1^{m_1}t_2^{m_2}$.  In particular, $\mu_0 \geq 0$ and $\mu_i > 0$ for all $i \geq 1$
as in \cite{HN}, which means that $\phi_i \in \mc{O}_K[[S]][u_i, v_i]$ is
divisible by
$u_i$ for $0 \leq i \leq L-1$ and by $v_i$ for $1 \leq i \leq L$.  


Let $\mc{X}'$ be obtained from gluing together the $\Spec
B_{\sigma_i}$ by identifying $u_i$ with $v_{i-1}^{-1}$ and $v_i$ with
$u_{i-1}v_{i-1}^{b_i}$ for $i \geq 1$.  The proof that $B_{\sigma_i}$ is a regular ring and
that $u_i$, $v_i$ form a regular system of parameters now goes through
verbatim as in the proof of \cite[Proposition~4.4.2.5]{HN}, as does the proof that the exceptional divisors $E_1,
\ldots, E_L$ are all isomorphic to $\proj^1_k$.  The divisor $E_i$ is
fully contained in $\Spec B_{\sigma_i} \cup \Spec B_{\sigma_{i-1}}$,
with $E_i \cap \Spec B_{\sigma_i} = \divi(v_i)$ and
$E_{i} \cap \Spec B_{\sigma_{i-1}} = \divi(u_{i-1})$. The point
$v_{i-1} = 0 \in E_i \cap \Spec B_{\sigma_{i-1}}$ is the unique point of
$E_i$ not contained in $\Spec B_{\sigma_i}$.
The fact that
$E_i$ meets $E_{i+1}$ transversely for $1 \leq i \leq L-1$ also
follows as in loc.\;cit.\;by ignoring the factors of $\phi_i$
other than powers of $u_i$ and $v_i$.  So $\mc{X}' \to \mc{X}$ is a
resolution of singularities whose exceptional divisor is a chain of
transversely meeting $\proj^1_k$'s.  It remains to show it is the
\emph{minimal} resolution,
which we do by showing that $(E_i, E_i) = -b_i \leq -2$.  This will prove (i) and (ii) together.

Let $\mc{X}_s'$ be the special fiber of $\mc{X}'$.  For
  $1 \leq i \leq L$, we begin by calculating the multiplicity of $E_i$ in
  $\mc{X}'_s$.  Let $\ol{\phi}_i \in k[[S]][u_i, v_i] \subseteq
  k[[u_i, v_i]]$ be the
  reduction of $\phi_i$ modulo $\pi_K$. By Definition~\ref{dfn: tame
    cyclic quotient admitting stable setting}, $\ol{\phi}_i$ is nonzero.  Let $wu_i^cv_i^d$ be the monomial
term of $\ol{\phi}_i \in k[[u_i, v_i]]$ for which
$d$ is minimal, and for which $c$ is maximal given this minimal $d$.\footnote{This
  maximum is attained since $u_i$, being dual to $e_i \in \rats_{>0}
  \times \rats_{>0}$ in the proof of
  \cite[Proposition 4.4.2.5]{HN}, has a negative power of either $t_1$
  or $t_2$ (in fact, $t_2$), so
 $c$ cannot be arbitrarily high given a fixed $d$.} Then there exists $\alpha \in
k[u_i]$ with degree less than $c$, as well as $\beta \in
k[[S]][u_i, v_i]$ divisible by $v_i$ 
such that 
$$\ol{\phi}_i = v_i^d(wu_i^c + \alpha + \beta).$$
Since
$\divi(v_i) = E_i \cap \Spec B_{\sigma_i}$, we have that the
multiplicity of $E_i$ in the special fiber is $d$.

Since $(E_i, \mc{X}'_s) = 0$, we have
$$(E_i, E_i) = -\frac{(E_i, \mc{X}_s' - dE_i)}{d}.$$  
We compute $(E_i, \mc{X}_s' - dE_i)$ as a sum of local intersection
numbers at closed points of $E_i$.  First, we observe that the contribution 
coming from closed points in $E_i \cap \Spec B_{\sigma_i}$ is the degree of $wu_i^c + \alpha + \beta$ as a polynomial in
  $u_i$, after setting $v_i = 0$.  This equals $c$.  It remains to
  compute the contribution coming from the unique closed point of
  $E_i$ not in $\Spec B_{\sigma_i}$.  As was mentioned above, this is
  the point $u_{i-1} = v_{i-1} = 0$
  in $\Spec B_{\sigma_{i-1}}$.

Since  $u_i
= v_{i-1}^{-1}$ and $v_i = u_{i-1}v_{i-1}^{b_i}$, we have
$wu_i^cv_i^d = wu_{i-1}^{c'}v_{i-1}^{d'}$, with $c' = d$ 
and $d' = db_i - c$.
Note that $c'$ is minimal among the $u_{i-1}$-exponents of the terms
of $\ol{\phi}_{i-1}$  and $d'$ is minimal among the
$v_{i-1}$-exponents given this minimal $c'$.  So there exist $\gamma \in k[[v_{i-1}]]$ divisible by $v_{i-1}^{d' +1}$ and $\delta \in
k[[S]][u_{i-1}, v_{i-1}]$ divisible by $u_{i-1}$ such that 
$$\ol{\phi}_{i-1} = u_{i-1}^{c'}(wv_{i-1}^{d'} + \gamma +
\delta).$$  So the contribution to $(E_i,
\mc{X}'_s - dE_i)$ coming
from this point is number of times $v_{i-1}$ divides $(wv_{i-1}^{d'} + \gamma +
\delta)$, after setting $u_{i-1} = 0$.  This equals $d'$.

Summing up, we have $(E_i, \mc{X}'_s - dE_i) = c +
d'$, which means that 
$$(E_i, E_i) = -\frac{c+d'}{d} = -\frac{c + db_i - c}{d} = -b_i. $$
\end{proof}

The following lemma is probably well-known, but we include a proof for
lack of a suitable reference. 
Recall that a simple normal crossings divisor $B$ on an arithmetic surface $\Y$ is a divisor all of
whose irreducible components are regular, and such that at most two irreducible
components of $B$ intersect at any point and all such intersections
are transverse. Similarly, a simple normal crossings point of $B$ is a point $y\in \Supp(B)$ that is a regular point of $B$ or is contained in exactly two irreducible components of $B$, which are regular and meet transversely at $y$. 


\begin{lemma}\label{Lnormalcrossings}
Every regular model of $\proj^1_K$ has simple normal crossings, and the dual graph of its special fiber is a tree.   
\end{lemma}

\begin{proof}

Let $\mc{Y}$ be any regular model of $\proj^1_K$ and let 
$\mc{Y}_s$ be the special fiber. Pick a $K$-rational point on $\proj^1_K$; it specializes to a
multiplicity $1$ component of $\mc{Y}_s$, say $\ol{V}$.  By
\cite[Lemma~7.1]{ObusWewers}, the
contraction of all components of $\mc{Y}_s$ other than $\ol{V}$ has
$\proj^1_k$ as its special fiber, so is a
smooth model $\mc{Y}^{sm}$ of $\proj^1_K$.  By
\cite[Theorem~9.2.2]{Liubook}, the contraction morphism $\mc{Y} \to
\mc{Y}^{sm}$ is made up of a finite sequence of blowups along closed
points.  Since the dual graph of the special fiber of $\mc{Y}^{sm}$
is a tree, the same holds for the dual graph of the special fiber of
$\mc{Y}$.  Furthermore, \cite[Lemma~9.2.31]{Liubook} shows that
$\mc{Y}$ has simple normal crossings.
\end{proof}

\begin{prop}\label{prop: structure of points upstairs}
    Let $\mc{Y}$ be regular model of $\P^1_K$,
    and let $D$ be the branch divisor for the normalization
    $\mc{X}\to\mc{Y}$ of $\mc{Y}$ in $K(X)$. Assume that the
    branch divisor $D$ is a simple normal crossings divisor.
    \begin{enumerate}[\upshape (i)]
        \item Every non-regular point of $\mc{X}$ lies over the intersection point of two irreducible components of $D$.
        \item Every closed point of $\mc{X}$ where the special fiber is not a simple normal crossings divisor lies over a horizontal irreducible component of $D$.
        \item\label{prop:whatappears} Let $y\in\mc{Y}$ be the intersection point of two
          irreducible components $D_1,D_2$ of $D$. For $i=1,2$, let
          $\nu_i \colonequals \ord_{D_i}(f), e_i \colonequals \frac{n}{\gcd(n,\nu_i)}$ and let $\kappa$ be the least residue mod $\gcd(e_1,e_2)$ of $-\left(\frac{\nu_1}{\gcd(\nu_1,n)}\right)^{-1}\frac{\nu_2}{\gcd(\nu_2,n)}$. There are
          $\frac{n}{\lcm(e_1,e_2)}$ points of $\mc{X}$ lying over $y$, and for such a point $x\in\mc{X}$, the scheme $\Spec \hat{\O}_{\mc{X},x}$ is a tame cyclic quotient singularity of type
        \[\frac{1}{\gcd(e_1,e_2)}(1,\kappa).\]
        \item Keep the notation of (iii) and suppose $D_1,D_2$ are contained in $\mc{Y}_s$. If $x\in\mc{X}$ lies over $y$, then the tame cyclic quotient singularity $\Spec \hat{\O}_{\mc{X},x}$ admits a stable setting.
    \end{enumerate}
  \end{prop}
\begin{proof}
Part (i) follows from \cite[Theorem 1.5]{KWsuper}, since the only
singular points of $D$ are those lying on two or more components.
Part (ii) follows from \cite[Proposition 2.8(i)]{OSNew}.

To prove the first assertion of part (iii), let $A_1$ and $A_2$ be irreducible components of
$\gamma^{-1}(D_1)$ and $\gamma^{-1}(D_2)$, respectively, passing
through a point $x$ above $y$.  By the definition of $e_1$ and $e_2$, the stabilizer of $x$ in $\ints/n$ contains the
group $H$ generated by $\ints/e_1$ and $\ints/e_2$ in $\ints/n$.
Since $H$ has
order lcm$(e_1, e_2)$, it suffices to show that $y$ is not a branch
point of $p \colon \mc{X}/H =: \mc{Z} \to \mc{Y}$.  Since $\mc{Y}$ is regular and
$\mc{X}$ is normal, this follows by purity of the branch locus (see, 
e.g., \cite[Theorem 5.2.13]{Szamuely}).

We now prove the second assertion of part (iii).  Since $p \colon \mc{Z} \to \mc{Y}$
is \'{e}tale over $y$, we can replace $y$ with any point $z$ of
$p^{-1}(y)$ and the $D_i$ with irreducible components of the
$p^{-1}(D_i)$ passing through $z$ without changing the $e_i$. So
we may assume $\mc{Z} = \mc{Y}$, or in other words, that lcm$(e_1, e_2) = n$.

Since $\hat{\mc{O}}_{\mc{Y}, y}$ is a regular local ring, 
the height one prime ideals corresponding to the restriction of $D_1$
and $D_2$ to $\hat{\mc{O}}_{\mc{Y}, y}$ are principal.  Let $s_1$ and
$s_2$ be their respective generators. By Hensel's lemma, all units in
$\hat{\mc{O}}_{\mc{Y}, y}$ are $n$th powers, so by Kummer theory, the extension
$\Frac(\hat{\mc{O}}_{\mc{Y}, y}) \hookrightarrow \Frac(\hat{\mc{O}}_{\mc{X}, x})$ is
given by an equation of the form
\begin{equation}\label{Es}
  s^n = s_1^{\nu_1}s_2^{\nu_2}.
\end{equation}
In fact, after replacing $s$
with an appropriate prime-to-$n$ power and multiplying it by
appropriate powers of $s_1$ and $s_2$, we may further assume that $\nu_1 =
n/e_1$ and $\nu_2 \in \{0, \ldots, n-1\}$, with $\nu_2$ divisible by
$n/e_2$.  On the other hand, since $\mc{Y}$ is regular with normal
crossings by Lemma~\ref{Lnormalcrossings}, we can write $\hat{\mc{O}}_{\mc{Y},y} \cong
\mc{O}_K[[u_1, u_2]]/(\pi_K - \beta u_1^au_2^b)$ for some $a, b \in
\ints_{\geq 0}$ and $\beta \in \mc{O}_K[[u_1,u_2]]^{\times}$.  Since $s_1$
and $s_2$ generate the maximal ideal of $\hat{\mc{O}}_{\mc{Y},y}$, there exist power series $\alpha_1$, resp.\ $\alpha_2$ 
over $\mc{O}_K$ such that $\alpha_1(s_1, s_2) =  u_1$ and
$\alpha_2(s_1, s_2) = u_2$.
Consider the ring
$$A \colonequals \hat{\mc{O}}_{\mc{Y}, y}[t_1,t_2]/(t_1^{e_1} - s_1, t_2^{e_2} -
s_2) \cong \mc{O}_K[[t_1, t_2]]/(\pi_K - \beta(t_1^{e_1},t_2^{e_2}) \alpha_1(t_1^{e_1},t_2^{e_2})^{a}\alpha_2(t_1^{e_1},t_2^{e_2})^{b}).$$  The maximal ideal of $A$ is
generated by $t_1$ and $t_2$, so $A$ is regular.  Furthermore,
$t_1^nt_2^{\nu_2e_2} = s_1^{\nu_1}s_2^{\nu_2} = s^n$ by \eqref{Es}.   So, up to
multiplication by an $n$th root of unity, we have
\begin{equation}\label{Ealternates}
  s = t_1t_2^{\nu_2e_2/n}.
\end{equation}

Let $\zeta$ be a primitive
$(e_1e_2/n)$-th root of unity and consider the $\ints/(e_1e_2/n)$-action on $A$ fixing $\mc{O}_K$ given by
$\sigma(t_1) = \zeta^{-\nu_2e_2/n}t_1$ and $\sigma(t_2) = \zeta t_2$.
By \eqref{Ealternates}, $\sigma$ fixes $s$, so by the normality of
$A^{\langle \sigma \rangle}$ and $\hat{\mc{O}}_{\mc{X}, x}$ we have
$A^{\langle \sigma \rangle} = \hat{\mc{O}}_{\mc{X}, x}$.  Thus
(switching the roles of $t_1$ and $t_2$), 
$\Spec \hat{\mc{O}}_{\X, x}$ is a tame cyclic quotient of type
$$\frac{1}{e_1e_2/n}(1, - \nu_2e_2/n).$$ Since $\lcm(e_1, e_2) = n$, we
have $e_1e_2/n = \gcd(e_1, e_2)$, so the type is $\frac{1}{\gcd(e_1,e_2)}(1, - \nu_2e_2/n)$.  Since
$\gcd(\nu_2,n) = n/e_2$ and $\nu_1 = n/e_1$, the proof of part (iii)
is complete.

Lastly, in the situation of part (iv), we have without loss of
generality that $u_1 = s_1$ and $u_2 = s_2$, so $A \cong
\mc{O}_K[[t_1,t_2]]/(\pi_K - t_1^{e_1a}t_2^{e_2b})$, implying
that $\hat{\mc{O}}_{\mc{X},x}$ admits a stable setting.
\end{proof}

Let $y,e_1,e_2,\kappa$ be as in Proposition~\ref{prop: structure of points upstairs}~\ref{prop:whatappears}. In light of Proposition~\ref{prop: structure of points upstairs}~\ref{prop:whatappears},
we make the following definition. 
\begin{defn}\label{dfn: Hirzebruch--Jung length of point}
    The \textit{Hirzebruch--Jung length of $y$} is the nonnegative integer
    \[\HJL(y)=L(\gcd(e_1,e_2),\kappa)\] where $L$ is the
    Hirzebruch--Jung length of a pair of integers from Definition
    \ref{dfn: Hirzebruch--Jung length of pair}.  
\end{defn}

\begin{remark}\label{rmk: bound on HJ length of point}
    Keep the notation of Definition \ref{dfn: Hirzebruch--Jung length
      of point}.  By Remark \ref{rmk: bound on HJ length of pair}, we
    have \[\HJL(y)\leq\gcd(e_1,e_2)-1. \]
    Furthermore, if $x
      \in \mc{X}$ lies above $y \in \mc{Y}$, then by
      Proposition~\ref{prop: resolution of tame cyclic quotient
        singularities}, the
    Hirzebruch--Jung length of $y$ is the number of
    irreducible components in the exceptional divisor of the
    minimal resolution of $\Spec \hat{\O}_{\X,x}$. By the proof of \cite[Proposition 3.5]{ObusWewers}, the same is true for the minimal resolution of $\X$ in $x$.
\end{remark}

\section{An inclusion--exclusion formula for the Artin conductor}\label{SEulerchar}


In this section, we first recall the definition of the Artin conductor.  Note that the definition works not only for regular models, but also for normal models. Then after recalling
some properties of the compactly supported Euler characteristic
function in \Cref{Peulerchar}, we use these to derive a formula
for the Artin conductor for certain special normal models
(Proposition~\ref{PartinNormalcalculation} and
Proposition~\ref{Partincalculation}).

Let $X$ be a smooth, projective, geometrically integral curve of genus $g \geq 1$ defined
over $K$.  Let $\mc{X}$ be a proper, flat, \emph{normal} $\mc{O}_K$-scheme with generic fiber $X$. 
\begin{defn}\label{DArtNormal}
The \emph{Artin conductor} associated to the model $\mc{X}$ is defined by
\[ \Art' (\mc{X}/\mc{O}_K) \colonequals  \chi(\mc{X}_{\overline{K}}) - \chi(\mc{X}_{k}) - \delta(X/K) ,\]
where $\chi$ is the Euler characteristic for the $\ell$-adic cohomology and $\delta(X/K)$ is the Swan conductor associated to the $\ell$-adic representation $\Gal (\overline{K}/K) \rightarrow \Aut_{\Q_\ell}
(H^1_{\mathrm{et}}(\mc{X}_{\overline{K}}, \Q_\ell))$ ($\ell \neq \cha k$).  Let \[\Art^+ (\mc{X}/\mc{O}_K) \colonequals -\Art' (\mc{X}/\mc{O}_K).\]
\end{defn}

\subsection*{The compactly supported Euler characteristic} Let $F$ be an algebraically closed field (in what follows, $F$ will be
$\ol{K}$ or $k$). Let $\mathcal{V}_F$ denote the set of isomorphism
classes of separated $F$-schemes of finite type. Let $\ell$ be a prime such that $\ell \neq \cha(F)$. For any
$F$-variety $X$ and integer $i \geq 0$, let $H^i_c(X,\Q_\ell)$ denote
the $i^{\mathrm{th}}$ $\ell$-adic \'{e}tale cohomology group with
compact support. The compactly supported Euler characteristic is the
function 
 \begin{align*} \chi \colon \mathcal{V}_F &\rightarrow \Z \\ X &\mapsto \sum_{i \geq 0} (-1)^i\dim H^i_c(X,\Q_{\ell}). \end{align*}
It has the following properties.
\begin{properties}\label{Peulerchar}\hfill 
\begin{enumerate}[label=(\arabic*)]
\item\label{Etopinv} \textbf{Topological invariance:}(cf. \cite[Chapter~1, Section~2, p.~22]{Milne}) For any $X \in \mathcal{V}_F$, we have $\chi(X) = \chi(X_{\mathrm{red}})$, where $X_{\mathrm{red}}$ is the underlying reduced subscheme of $X$.
\item\label{Eadditivity} \textbf{Additivity:}(cf. \cite[Chapter~1, Section~9, p.~65]{Milne}) If $X = \bigsqcup Z_i$ is a decomposition of $X$ into a disjoint union of locally closed subschemes, then $\chi(X) = \sum \chi(Z_i)$.
\item\label{Epoint} \textbf{Dimension axiom:}(cf. \cite[Chapter~1, Section~9, p.~65]{Milne}) $\chi(\Spec(F)) = 1$.
\item\label{Echain} \textbf{Cohomology of the projective line:}(cf. \cite[Chapter~1, Section~14, p.~90]{Milne})
  $\chi(\P^1_{\overline{F}}) = 2$. Combined with the dimension axiom
  and additivity, it follows that the compactly supported Euler
  characteristic of a tree of $N$ projective lines over $F$ is $N+1$.
\item\label{ERHformula} \textbf{\'{E}tale Riemann-Hurwitz formula:}(cf. \cite[Chapter~V, Remark~2.15~(c)]{MilnePrint}) Let $Y$ be a
  regular $F$-scheme and let $f \colon X \rightarrow Y$ be a finite
  \'{e}tale map of $F$-schemes of degree $d$.  Then $\chi(X) = d
  \chi(Y)$.
\end{enumerate}
\end{properties}


\begin{example}\label{example:optimal} Let $X$ be the curve with affine equation $y^n=x^n-\pi_K$. The geometric generic fiber $\X_{\c{K}}$ is smooth of genus $g=\frac{(n-2)(n-1)}{2}$, and the special fiber $\X_k$ consists of $n$ irreducible components isomorphic to $\P^1_k$ meeting pairwise transversely at a single point. By Properties \ref{Peulerchar}, we see $\chi(\X_{\c{K}})=3n-n^2$ and $\chi(\X_k)=n+1$, and by Proposition \ref{Pswanconductor}, the Swan conductor $\delta(X/K)$ of $X$ is $0$, so that $\Art^+(X/K)=(n-1)^2$. 
\end{example}

\subsection{Artin conductor of the normalization of a regular model of $\proj^1_K$
  with normal crossings branch divisor}\label{Sfirstartinformula}

For the rest of \S\ref{Srephrasing}, let $X$ be the smooth projective $K$-curve with affine
equation $y^n = f(x)$ satisfying \Cref{Afform}.
In \S\ref{Sfirstartinformula}, we make the following assumption:

\begin{assumption}\label{Ahypoonmodel}
We are given a regular model $\mc{Y}$ of $\P^1_K$ whose normalization
$\mc{X} \to \mc{Y}$ in $K(X) = K(x)[y]/(y^n-f(x))$ has branch locus $B$ that
is a simple normal crossings divisor.
  \end{assumption}
 
 \begin{notation}\label{Nmodel} Under Assumption~\ref{Ahypoonmodel},
   let $R$ be the inverse image of $B$ in $\mc{X}$.
  \begin{itemize}
 \item Write $\mc{Y}_s$ and $\mc{Y}_{\ol{\eta}} \cong \proj^1_{\ol{K}}$ for the special and geometric generic fibers of $\mc{Y}$, respectively. Similarly define $B_s,B_{\ol{\eta}},R_s$ and $R_{\ol{\eta}}$.
 \item For every $e \mid n$, let $N_e$ be the number of irreducible components of $\mc{Y}_s$ with ramification index $e$. Let $N = \sum_{e \mid n} N_e$. 
 \item For every pair of integers $e,e'$ dividing $n$, let $c_{ee'}$
   be the number of pairwise intersections of an irreducible component
   of $B$ with ramification index $e$ with
   another such component of ramification index $e'$. 
\end{itemize}


Fix $f$ as in \Cref{Afform}. For each $1 \leq i \leq
r$, recall that we have also chosen a root $\alpha_i$ of $f_i$ and we set $K_i = K(\alpha_i)$.
 \end{notation}

\begin{prop}\label{PartinNormalcalculation} Work under Assumption~\ref{Ahypoonmodel} and Notation~\ref{Nmodel}.
  Then $\Art^+(\mc{X}/\mc{O}_K)$ equals
$$ n \left(N-1 \right)-\sum_{e \mid n} \left(n-\frac{n}{e} \right) \left[2N_e -  \sum_{\substack{i \\ \gcd(d_i,n) = n/e}} \Delta_{K_i/K} \right]+\sum_{\substack{e \leq e' \\ e \mid n, e' \mid n}} c_{ee'} \left[ n+\gcd \left(\frac{n}{e},\frac{n}{e'} \right)-\frac{n}{e}-\frac{n}{e'} \right].$$
\end{prop}
\begin{proof}
From Definition~\ref{DArtNormal} and Proposition~\ref{Pswanconductor}, it suffices to show that $\chi(\mc{X}_s)-\chi(\mc{X}_{\overline{K}})$ equals
\begin{equation}\label{Etoprove}
\begin{aligned}
n \left(N-1 \right)-\sum_{e \mid n} \left(n-\frac{n}{e} \right) \left[2N_e -  \sum_{\substack{i \\ \gcd(d_i,n) = \frac{n}{e}}} (\deg(f_i)-1) \right]
\\+\sum_{\substack{e \leq e' \\ e \mid n, e' \mid n}} c_{ee'} \left[ n+\gcd \left(\frac{n}{e},\frac{n}{e'} \right)-\frac{n}{e}-\frac{n}{e'} \right] 
\end{aligned}
\end{equation}

To prove this, we will use Properties~\ref{Peulerchar} of $\chi$ with $F=\overline{K}$ and $F=k$. (We will omit specifying which of the two cases we are in since it is clear from context.)

Since $\cha(k) \nmid n$, it follows that $\mc{X}_s \setminus R_s \rightarrow \mc{Y}_s \setminus B_s$ and $\mc{X}_{\overline{\eta}} \setminus R_{\overline{\eta}} \rightarrow \mc{Y}_{\overline{\eta}} \setminus B_{\overline{\eta}}$ are finite \'{e}tale degree $n$ maps. For $*$ in $\{ s, \overline{\eta} \}$, using Property \ref{Peulerchar}~\ref{ERHformula} we get
\begin{equation*}
 \chi(\mc{X}_* \setminus R_*) = n \chi(\mc{Y}_* \setminus B_*), 
\end{equation*}
and combining this with Property \ref{Peulerchar}~\ref{Eadditivity} applied to the decompositions $\mc{X}_* = R_* \sqcup (\mc{X}_* \setminus R_*)$ and $\mc{Y}_* = B_* \sqcup (\mc{Y}_* \setminus B_*)$, we get
\begin{align}\label{Eeulerdiff1}
\begin{split}
 \chi(\mc{X}_s)-\chi(\mc{X}_{\overline{K}}) &= \chi(\mc{X}_s \setminus R_s)+\chi(R_s)-\chi(\mc{X}_{\overline{\eta}} \setminus R_{\overline{\eta}})-\chi(R_{\overline{\eta}}) \\
 &= n(\chi(\mc{Y}_s \setminus B_s) - \chi(\mc{Y}_{\overline{\eta}} \setminus B_{\overline{\eta}}))+\chi(R_s)-\chi(R_{\overline{\eta}}) \\
 &= n(\chi(\mc{Y}_s)-\chi(B_s) - \chi(\mc{Y}_{\overline{\eta}})+\chi(B_{\overline{\eta}}))+\chi(R_s)-\chi(R_{\overline{\eta}}) \\ 
 &=  n(\chi(\mc{Y}_s)-\chi(\mc{Y}_{\overline{\eta}}))+\left[ n\chi(B_{\overline{\eta}}) -\chi(R_{\overline{\eta}}) \right]+ \left[-n\chi(B_s)+\chi(R_s)\right].
\end{split}
 \end{align}
Since $\mc{Y}_{\overline{\eta}} \cong \P^1_{\overline{K}}$ and
$\mc{Y}_s$ is a tree of $N$ (possibly
non-reduced) projective lines over $k$, using Properties \ref{Peulerchar}~\ref{Etopinv} and \ref{Echain}, we get
\begin{equation}\label{Echinumberofcomp}
 \chi(\mc{Y}_s)-\chi(\mc{Y}_{\overline{\eta}}) = \left( N+1 \right)-2 = N-1.
\end{equation}

 Let $B^\circ \subset B$ be the open subset of the branch locus such
 that its complement is all pairwise intersections of components of
 $B$, and let $R^\circ$ be the preimage of $B^\circ$ under $\mc{X}
 \rightarrow \mc{Y}$. Then we have $B^\circ = \sqcup_{e \mid n}
 B^\circ_e$, where $B^\circ_e \subset B^\circ$ is the union of
 irreducible components of ramification index $e$. Let $R^\circ_e$ be
 the preimage of $B^\circ_e$ under $\mc{X} \rightarrow \mc{Y}$.
 Since, for all $e \mid n$,
 $R^{\circ}_e \to B^{\circ}_e$ is a topologically trivial map of
 degree $e$ followed by an \'{e}tale map of degree $n/e$, Properties \ref{Peulerchar}~\ref{Etopinv} and
 \ref{ERHformula} show that
 \begin{equation}\label{Echiramdiv1}
  \chi((R^\circ_e)_s) = \frac{n}{e} \chi((B^\circ_e)_s), \end{equation} and
  \begin{equation}\label{Echiramdiv2} \chi((R^\circ_e)_{\overline{\eta}}) = \frac{n}{e} \chi((B^\circ_e)_{\overline{\eta}}).
 \end{equation}
Since $n \mid \deg(f)$, it follows that $B_{\overline{\eta}}$ is
precisely the zeroes of $f$ and using Property
\ref{Peulerchar}~\ref{Epoint} and \Cref{Afform} we have 
\begin{equation}\label{Echibranchgeneric0}  
\chi((B_e)_{\overline{\eta}}) = \sum_{i,\gcd(d_i,n) = n/e} \deg(f_i), \end{equation} and, 
\begin{equation}\label{Echibranchgeneric}  
\chi(B_{\overline{\eta}}) = \sum_{e \mid n} \sum_{i,\gcd(d_i,n) = n/e} \deg(f_i). \end{equation}
Furthermore, since $B^\circ_{\overline{\eta}} = B_{\overline{\eta}}$, it also follows that $R^\circ_{\overline{\eta}} = R_{\overline{\eta}}$, and combining Equations~\eqref{Echiramdiv2} and \eqref{Echibranchgeneric0} and adding over all $e \mid n$, we also get
\begin{equation}\label{Echiramdiv3}
 -\chi(R_{\overline{\eta}}) = \sum_{e \mid n} \frac{n}{e}\left(-\sum_{i,\gcd(d_i,n) = n/e} \deg(f_i) \right).
\end{equation}

Now  $(B_e)_s$ is the union of $N_e$ vertical components of
ramification index $e$, each isomorphic to $\P^1_k$, and closed
points, each isomorphic to $\Spec k$, one for each of the
specializations of the $f_i$ that satisfy $\gcd(d_i,n) = n/e$. Using
Property \ref{Peulerchar}~\ref{Epoint}, we see that the sum over all components
of $(B_e)_s$ of the Euler characteristic of the intersection of each
component of $(B_e)_s$ with components of $B_s$ is $2c_{ee}
+\sum_{\substack{e' \neq e}}  c_{ee'}$. Combining this with
Property \ref{Peulerchar}~\ref{Eadditivity} applied to this decomposition
and Properties \ref{Peulerchar}~\ref{Epoint} and
\ref{Peulerchar}~\ref{Echain}, we get
\begin{equation}\label{Echibranchspecial} \chi((B^\circ_e)_s) = 2N_e -2c_{ee} -\sum_{\substack{e' \neq e}}  c_{ee'} +\sum_{\substack{i,\gcd(d_i,n) = n/e}} 1.\end{equation}
Combining Equations~\eqref{Echiramdiv1} and \eqref{Echibranchspecial} and adding over all $e \mid n$, we also get
\begin{equation}\label{Echiramdivint}
 -n\chi((B^\circ)_s)+\chi((R^\circ)_s) = \sum_{e \mid n} \left(-n+ \frac{n}{e} \right) \left[ 2N_e -2c_{ee} -\sum_{\substack{e' \neq e}}  c_{ee'} +\sum_{\substack{i,\gcd(d_i,n) = n/e}} 1 \right].
\end{equation}
Now for every closed point of $B$ that lies on the intersection of a component of ramification indices $e$ and $e'$, by Proposition~\ref{prop: structure of points upstairs}(iii) we know that there are $\gcd(n/e,n/e')$ closed points above it in $R$. This implies that
\[ -n\chi(B_s \setminus B^\circ_s) = \sum_{\substack{e \leq e' \\ e \mid n, e' \mid n}} -nc_{ee'}, \text{ and, }\chi(R_s \setminus R^\circ_s) = \sum_{\substack{e \leq e' \\ e \mid n, e' \mid n}} c_{ee'} \gcd \left(\frac{n}{e},\frac{n}{e'} \right),\]
and therefore
\begin{equation}\label{Echibranchintersection}
 -n\chi(B_s \setminus B^\circ_s) + \chi(R_s \setminus R^\circ_s) = \sum_{\substack{e \leq e' \\ e \mid n, e' \mid n}} c_{ee'} \left( -n+\gcd \left(\frac{n}{e},\frac{n}{e'} \right) \right).
\end{equation}
Adding $n$ times Equations~\eqref{Echinumberofcomp} and \eqref{Echibranchgeneric} to Equations~\eqref{Echiramdiv3},~\eqref{Echiramdivint} and ~\eqref{Echibranchintersection}, applying Property \ref{Peulerchar}~\ref{Eadditivity} to the decompositions $B_s = B_s^\circ \sqcup (B_s \setminus B_s^{\circ}), R_s = R_s^\circ \sqcup (R_s \setminus R_s^{\circ})$
and rearranging the right hand side using the equality 
\[\sum_{e \mid n} \left(n - \frac{n}{e}\right) \left(2c_{ee} + \sum_{e' \neq e} c_{ee'} \right) + \sum_{e \leq e'} c_{ee'} \left( -n+\gcd \left(\frac{n}{e},\frac{n}{e'} \right) \right) \] \[= \sum_{\substack{e \leq e' \\ e \mid n, e' \mid n}} c_{ee'} \left[ n+\gcd \left(\frac{n}{e},\frac{n}{e'} \right)-\frac{n}{e}-\frac{n}{e'} \right], \]
we get that the right hand side of Equation~\eqref{Eeulerdiff1} equals the expression in (\ref{Etoprove}).\qedhere


\end{proof}

Let $\Y$ be as in Assumption \ref{Ahypoonmodel}. Let $\X\to\Y$ be the
 normalization of $\Y$ in $K(X)$. Let $z_1,...,z_l$ be the points of multiplicity $2$ in the branch divisor $B$ of $\X\to\Y$. For $1\leq i\leq l$, let $e_i,e_i'$ denote the ramification indices of the two irreducible components of $B$ meeting at $z_i$.
\begin{prop}\label{Partincalculation}
Let $\mc{X}'\to\X$ be the minimal regular resolution of $\X$. Then  
  \[ \Art^+(\mc{X}'/\mc{O}_K)-\Art^+(\mc{X}/\mc{O}_K) = \sum_{i=1}^l 
     \gcd \left(\frac{n}{e_i},\frac{n}{e_i'} \right)
     \cdot\HJL(z_i).\]
\end{prop}
\begin{proof}
Let $T \subset \mc{X}_s$ be the set of singular points of $\mc{X}$ and
let $T' \subset \mc{X}'_s$ be the preimage of the singular points
under $\mc{X}' \rightarrow \mc{X}$. Since the map $\mc{X}' \rightarrow
\mc{X}$ induces isomorphisms $\mc{X}'_{\overline{\eta}} \rightarrow
\mc{X}_{\overline{\eta}}$ and $\mc{X}'_s \setminus T' \rightarrow
\mc{X}_s \setminus T$, from Property \ref{Peulerchar}~\ref{Eadditivity} applied to the decompositions $\mc{X}_s = T \sqcup (\mc{X}_s \setminus T)$ and $\mc{X}'_s = T' \sqcup (\mc{X}'_s \setminus T')$, it suffices to show that 
\[ \chi(T') - \chi(T) = \sum_{i=1}^l 
     \gcd \left(\frac{n}{e_i},\frac{n}{e_i'} \right)
     \cdot\HJL(z_i) \]

 We now describe $T$ and $T'$. By Proposition~\ref{prop: structure of points upstairs}(i,iii), the
 singularities of $\mc{X}$ are isolated, and lie above the points
 $z_i$. Furthermore by Remark \ref{rmk: bound on HJ length of point},
 if $x_i$ is a point above $z_i$, then the reduced exceptional divisor
 of the minimal resolution of $x_i$ is a
 chain of $\HJL(z_i)$ components isomorphic to $\P^1_k$, and by Proposition \ref{prop: structure of points upstairs}(iii), there
 are exactly $\gcd \left(\frac{n}{e_i},\frac{n}{e_i'} \right)$ such
 points $x_i$ for each $z_i$. The proposition now follows from
 Properties~\ref{Peulerchar}~\ref{Eadditivity} and \ref{Echain}.
\end{proof}

\section{Rephrasing the conductor-discriminant inequality}\label{Srephrasing}
The goal of this section is to prove \Cref{prop: cond-disc ineq cdc version}, where we use Proposition~\ref{PartinNormalcalculation} and
Proposition~\ref{Partincalculation} to
rephrase the conductor-discriminant inequality \eqref{Econddiscineq}
in terms of a quantity called the \emph{conductor
  discriminant contribution} (cdc) attached to the resolution of each
isolated singularity on a normal model of $X$ (see \Cref{dfn: cdc of sequence} and
Proposition~\ref{prop: cond-disc ineq cdc version}).  Similarly, in
Corollary~\ref{cor: exp-disc ineq cedc version}, we rephrase the
conductor exponent-discriminant inequality in terms of a quantity
called the \emph{conductor exponent discriminant contribution} (cedc).

\subsection{Conductor-discriminant contribution of a sequence of blowups}

For the rest of \S\ref{Srephrasing}, we drop
Assumption~\ref{Ahypoonmodel} but still assume $X$ is given by an equation $y^n = f(x)$ as in \Cref{Afform}.

\begin{notation}\label{notation: divisor notations}(Branch divisor)
    Let $\Y$ be a regular model of $\P^1_{K}$. Let $\X$ be the normalization of $\Y$ in $K(X)$. If $\Gamma$ is an irreducible divisor on $\Y$, the ramification index of $\Gamma$ always means the ramification index with respect to $\X \rightarrow \Y$; likewise, the branch divisor on $\Y$ always means the branch divisor of $\X \rightarrow \Y$.  
    
    Finally, for a divisor $E=\sum a_i\Gamma_i$ on $\Y$, let $(E\tx{ mod }n)^{\red}$ denote the divisor $\sum b_i\Gamma_i$ on $\Y$, where $b_i=0$ if $n\mid a_i$ and $b_i=1$ otherwise. For example, if $E=\divi_0(f)$ on $\Y$, then $(E\tx{ mod }n)^{\red}$ is the branch divisor of $\X\to\Y$.
\end{notation}

\begin{notation}\label{not:multiplicities}(Multiplicities and ramification indices)  Let $\Y$ be a regular model of $\P^1_K$, and let $F$ be an effective divisor on $\Y$. 
    Let $\Y_m\to\Y_{m-1}\to...\to\Y_0=\Y$ be a finite sequence of blowups at reduced closed points. For $0\leq k\leq m$, let $F_k$ be the total transform of $F$ under $\Y_k\to\Y_0$, and let $D_k=(F_k\tx{ mod }n)^{\red}$. For $1\leq k\leq m$, let $y_k\in\Y_{k-1}$ denote the center of the blowup $\Y_k\to\Y_{k-1}$, and let $\nu_k \colonequals \mult_{y_k} (F_{k-1})$ and 
    $\mu_k \colonequals \mult_{y_k} (D_{k-1})$. 
    Let $z_1,...,z_l$ be the multiplicity $2$ simple normal crossings points of the branch divisor on $\Y_m$ that are contained in $D_m$, and for $1\leq i\leq l$, let $e_i,e_i'$ be the ramification indices of the two irreducible components of the branch divisor meeting at $z_i$.
\end{notation}

\begin{remark}
    Keep Notation \ref{not:multiplicities}. If $F=\divi_0(f)$, so that $D_m$ is equal to the branch divisor on $\Y_m$, or if $F=F_y$ as in Notation \ref{NSection10}, so that Lemma \ref{lemma: compatible resolutions}(ii) applies, the points $z_1,...,z_l$ are simply the multiplicity $2$ simple normal crossings points of $D_m$. These are the only cases of interest to us.
\end{remark}

\begin{defn}\label{dfn: cdc of sequence}
    Let $\mc{Y}_i,z_i,\nu_k,\mu_k,e_i$ be as in \Cref{not:multiplicities}. The \textit{conductor-discriminant contribution} 
    associated to the sequence of blowups $\Y_m\to\Y_{m-1}\to...\to\Y_0$ and divisor $F$ is defined by
    \begin{align*}
        \cdc(\Y_m\to\Y_{m-1}\to...\to\Y_0,F)\colonequals&\sum_{ k:n\nmid\nu_k}(n-2\gcd(\nu_k,n)+(n-1)\mu_k(\mu_k-3))\\
        +&\sum_{k:n\mid\nu_k}(-n+(n-1)\mu_k(\mu_k-1))\\+&\sum_{i=1}^l\left(\left(\frac{n}{e_i}-1\right)+\left(\frac{n}{e_i'}-1\right)+\left(n-\gcd \left(\frac{n}{e_i},\frac{n}{e_i'} \right)(\HJL(z_i)+1)\right)\right).
    \end{align*}
    Similarly, the \textit{conductor exponent-discriminant contribution} 
    is defined by
    \begin{align*}
        \cedc(\Y_m\to\Y_{m-1}\to...\to\Y_0,F)\colonequals&\sum_{ k:n\nmid\nu_k}(1+n-2\gcd(\nu_k,n)+(n-1)\mu_k(\mu_k-3))\\
        +&\sum_{k:n\mid\nu_k}(1-n+(n-1)\mu_k(\mu_k-1))
        \\+&\sum_{i=1}^l\left(\left(\frac{n}{e_i}-1\right)+\left(\frac{n}{e_i'}-1\right)+\left(n-\gcd \left(\frac{n}{e_i},\frac{n}{e_i'} \right)\right)\right).
    \end{align*}

\end{defn}

This terminology is motivated by the following proposition.

\begin{prop}\label{prop: cond-disc ineq cdc version} Let $\Y=\P^1_{\O_K}$, and let $\Y_0,\Y_m,\X_m$ be as in \Cref{not:multiplicities}. Let $F \colonequals \divi_0(f) \subset \Y_0$. 
    Assume that the branch divisor of the normalization $\X_m\to\Y_m$
    of $\Y_m$ in $K(X)$ has simple normal crossings. If
    $\mathcal{X}_m' \to \mathcal{X}_m$ is the minimal regular
    resolution, then
    \begin{align*}
        (n-1)\nu_K(\disc_n(f))-\Art^+(\mathcal{X}_m')&=\cdc(\Y_m\to\Y_{m-1}\to...\to\Y_0, F)\\
        &+\sum_{i=1}^r(\gcd(n,d_i)-1)\Delta_{K_i/K}+\begin{cases}0&b=0\\2-2\gcd(b,n)&b\geq1 \end{cases}\\
        &\geq \cdc(\Y_m\to\Y_{m-1}\to...\to\Y_0, F)+\begin{cases}0&b=0\\2-2\gcd(b,n)&b\geq1. \end{cases}
    \end{align*}
\end{prop}


We first prove a few lemmas. Let $\Y$ be a regular model of $\P^1_K$. Let $D$ be the branch divisor of the normalization $\X\to\Y$ of $\Y$ in $K(X)$. Let $y\in\Y$ be a closed point such that $\mult_y D = \mu$ and $\mult_y (\divi_0(f)) = \nu$.
Let $\pi \colon \Y'\to\Y$ be the blowup of $\Y$ at $y$, let $E$ be the corresponding exceptional divisor with normalized discrete valuation $\ord_E \colon K(Y)^\times \rightarrow \Z$, and let $D'$ be the branch divisor of the normalization $\X'\to\Y'$ of $\Y'$ in $K(X)$. 
Let $\tilde{D}$ and $\tilde{D}'$ denote the normalizations of the divisors $D$ and $D'$ respectively. We continue to use $D$ to denote the strict transform in $\mc{Y}'$ of the divisor $D \subset\mc{Y}$.
\begin{lemma}\label{lemma: length change with blowup}
\hfill
\begin{enumerate}[\upshape (i)] 
\item\label{L:multexc} $\ord_E(f) = \nu$. If $n\mid\nu$, then $D'=D$, and if $n\nmid\nu$, then $D'=D+E$.
\item   \[\l{\tilde{D}}{D}=\l{\tilde{D'}}{D'}+\begin{cases}
        \frac{\mu(\mu-1)}{2}&n\mid\nu\\
        \frac{\mu(\mu-3)}{2}&n\nmid\nu.\end{cases}\]
\end{enumerate}
\end{lemma}
\begin{proof} By definition of $X, D,$ and $D'$, it follows that $D \equiv (\divi_0(f)\tx{ mod }n)^{\red}\subset \Y$ and $D' \equiv (\divi_0(f)\tx{ mod }n)^{\red}\subset\Y'$. By \cite[Proposition~9.2.23]{Liubook}, it follows that $\ord_E(f) = \mult_y (\divi_0(f)) \equalscolon \nu$.
From the last two sentences, it follows that if $n\mid\nu$, then $D'$ is the strict transform of $D$, and if $n\nmid\nu$, then $D'=D+E$, where we continue to use $D$ to denote the strict transform in $\mc{Y}'$ of the divisor $D \subset \mc{Y}$. This proves \ref{L:multexc}. Part (ii) of the lemma then follows from \cite[Equations (3.6) and (3.7)]{OShyper_published}, where we ignore the final inequalities in both equations, which are the only parts of these equations that depend on the assumption $n=2$ in \cite{OShyper_published}.   
\end{proof}

\begin{corollary}\label{Elengthexpression}
 Let $D_0,D_m$ be as in \Cref{not:multiplicities}. Assume that $D_m$ has simple normal crossings and that $l$ is the number of multiplicity $2$ points of $D_m$. Then
 \[ \l{\tilde{D}_0}{D_0}=l+\sum_{k:n\nmid\nu_k}\frac{\mu_k(\mu_k-3)}{2}+\sum_{k:n\mid\nu_k}\frac{\mu_k(\mu_k-1)}{2}.\]
\end{corollary}
\begin{proof}
Since $D_m$ has simple normal crossings and $l$ is the number of multiplicity $2$ points of $D_m$, it follows that $\l{\tilde{D}_m}{D_m} = l$. Combining this with a repeated application of Lemma \ref{lemma: length change with blowup}, we get
\begin{equation*}
    \begin{aligned}
        \l{\tilde{D}_0}{D_0}&=\l{\tilde{D}_m}{D_m}+\sum_{k:n\nmid\nu_k}\frac{\mu_k(\mu_k-3)}{2}+\sum_{k:n\mid\nu_k}\frac{\mu_k(\mu_k-1)}{2}\\
        &=l+\sum_{k:n\nmid\nu_k}\frac{\mu_k(\mu_k-3)}{2}+\sum_{k:n\mid\nu_k}\frac{\mu_k(\mu_k-1)}{2}.
      \end{aligned} 
      \end{equation*}
\end{proof}

\begin{lemma}\label{Edbexpression}
Suppose $\Y=\P^1_{\O_K}$, and let $D_0$ be as in \Cref{not:multiplicities}. Then
 \[ \db_K(f) =2\l{\tilde{D}_0}{D_0}-\begin{cases}
     0&b=0\\
     2&b\geq 1
 \end{cases}\qedhere\] 
\end{lemma}
\begin{proof}
For convenience, set $c=0$ if $b=0$ and $c=1$ otherwise. Let $D_0'$ be the
 horizontal part of $D_0$, and let $\widetilde{D_0'}$ denote the normalization of $D_0'$. In \eqref{Edbexpresscorrect} below, the first equality follows from
    \Cref{Rdiscprop}, the second from the equation between (3.9) and
    (3.10) in \cite{OShyper_published}, and the third from
    \cite[(3.9)]{OShyper_published}.

    \begin{equation}\label{Edbexpresscorrect}
    \begin{aligned}
        \db_K(f)&=2c(\deg(\rad(f))-1)+\db_K(f/\pi^b) & \\
        &=2c(\deg(\rad(f))-1)+2\tx{ }\l{\widetilde{D_0'}}{D_0'} & \\
        &=2c(\deg(\rad(f))-1)+2\tx{ }\l{\widetilde{D}_0}{D_0}-2c\deg(\rad(f)) & \\
        &=2\tx{ }\l{\tilde{D}_0}{D_0}-2c.\quad \quad \quad \quad \quad \quad &\qedhere
      \end{aligned}
      \end{equation}
      
\end{proof}


\begin{proof}[Proof of Proposition \ref{prop: cond-disc ineq cdc version}]
Keep the notation of \Cref{dfn: cdc of sequence}.  For $e\mid n$, let $N_e$ be the number of irreducible components of the special fiber $(\Y_m)_s$ of ramification index $e$, and let $N \colonequals \sum_{e\mid n}N_e$.
By Propositions \ref{PartinNormalcalculation}, \ref{Partincalculation} and \ref{prop: structure of points upstairs}(iii), we have
\begin{equation}\label{eq:artplus}
    \begin{aligned}
        \Art^+(\X'_m&)
        =n(N-1)-\sum_{e\mid n} \left(n-\frac{n}{e} \right) \left(2N_e-\sum_{i:\gcd(d_i,n)=\frac{n}{e}}\Delta_{K_i/K} \right)\\
        &\quad\quad+\sum_{i=1}
        ^l \left(n+\gcd \left(\frac{n}{e_i},\frac{n}{e_i'} \right)-\frac{n}{e_i}-\frac{n}{e_i'} \right)+\sum_{i=1}^l\gcd \left(\frac{n}{e_i},\frac{n}{e_i'} \right)\HJL(z_i). 
    \end{aligned}
    \end{equation}

    The special fiber $(\mc{Y}_0)_s=\P^1_k$ is irreducible and has
    ramification index $\frac{n}{\gcd(n,b)}$, where $b$ is the
    valuation of the leading coefficient of $f$ as in Assumption~\ref{Afform}. By \Cref{lemma: length change with blowup}~\ref{L:multexc}, the blowup $\Y_k\to\Y_{k-1}$, centered at a closed point in $\Y_{k-1}$ of multiplicity $\nu_k$ in $\divi_0(f)$, introduces a single irreducible component in the special fiber on which the order of $f$ is $\nu_k$, which implies that the ramification index of this component is $\frac{n}{\gcd(n,\nu_k)}$. It follows that $N=m+1$,
    \begin{align*}
        \sum_{e\mid n}2N_e \left(n-\frac{n}{e} \right)&=2(n-\gcd(n,b))+\sum_{k=1}^m2(n-\gcd(\nu_k,n)),
    \end{align*}
    and therefore
    \begin{equation}\label{eq:compcount}
    \begin{aligned}
        n(N-1)-\sum_{e\mid n}2N_e \left( n-\frac{n}{e} \right)&=-2(n-\gcd(n,b))+ \sum_{k=1}^{m}(2\gcd(\nu_k,n)-n).
    \end{aligned}
    \end{equation}

    We also have
    \begin{equation}\label{eq:disc}
    \begin{aligned}
        \sum_{e\mid n} \left(n-\frac{n}{e} \right) \left(\sum_{i:\gcd(d_i,n)=\frac{n}{e}}\Delta_{K_i/K} \right)&=\sum_{i=1}^r(n-\gcd(n,d_i))\Delta_{K_i/K}
        \\&=(n-1)\sum_{i=1}^r\Delta_{K_i/K}-\sum_{i=1}^r(\gcd(n,d_i)-1)\Delta_{K_i/K}.
    \end{aligned}
    \end{equation}
    
    Putting \eqref{eq:artplus}, \eqref{eq:compcount}, and \eqref{eq:disc} together, we get

   \begin{equation}\label{EArtexpression}
    \begin{aligned}
        \Art^+(\X'_m) &=-2(n-\gcd(n,b))+\sum_{k=1}^m(2\gcd(\nu_k,n)-n)\\&+(n-1)\sum_{i=1}^r\Delta_{K_i/K}-\sum_{i=1}^r(\gcd(d_i,n)-1)\Delta_{K_i/K}
        \\&+\sum_{i=1}^l\left(n-\frac{n}{e_i}-\frac{n}{e_i'}+\gcd \left(\frac{n}{e_i},\frac{n}{e_i'} \right)(\HJL(z_i)+1)\right).
      \end{aligned}
      \end{equation}

   Using \eqref{EArtexpression} and \Cref{Ddiscbonus} for the first
   equality, \Cref{Elengthexpression} and
   \Cref{Edbexpression} for the second equality, and \Cref{dfn: cdc of sequence} for the final equality, we get that
   \begin{align*}
   \quad &(n-1)\nu_K(\disc_n(f))-\Art^+(\mathcal{X}_m') \\        
        \quad= &\ (n-1)\db_K(f)+2(n-\gcd(n,b))-\sum_{k=1}^m(2\gcd(n,\nu_k)-n)+\sum_{i=1}^r(\gcd(n,d_i)-1)\Delta_{K_i/K}\\
        \quad &-\sum_{i=1}^l\left(n-\frac{n}{e_i}-\frac{n}{e_i'}+\gcd \left(\frac{n}{e_i},\frac{n}{e_i'} \right)(\HJL(z_i)+1)\right)\\
        =  &\ 2l(n-1)+\sum_{k:n\nmid\nu_k}(n-1)\mu_k(\mu_k-3)+\sum_{k:n\mid\nu_k}(n-1)\mu_k(\mu_k-1)-\begin{cases}
            0&b=0\\
            2(n-1)&b\geq 1
        \end{cases}\\
        \quad\quad & +2(n-\gcd(n,b))-\sum_{k=1}^m(2\gcd(n,\nu_k)-n)+\sum_{i=1}^r(\gcd(n,d_i)-1)\Delta_{K_i/K}\\
        \quad\quad & -\sum_{i=1}^l\left(n-\frac{n}{e_i}-\frac{n}{e_i'}+\gcd \left(\frac{n}{e_i},\frac{n}{e_i'} \right)(\HJL(z_i)+1)\right)\\
        = &\sum_{k:n\nmid\nu_k}\left(n-2\gcd(\nu_k,n)+(n-1)\mu_k(\mu_k-3)\right)+\sum_{k:n\mid\nu_k}\left(-n+(n-1)\mu_k(\mu_k-1)\right)\\
        \quad\quad & +\sum_{i=1}^l\left(\frac{n}{e_i}-1 \right)+\left(\frac{n}{e_i'}-1 \right)+ \left(n-\gcd \left(\frac{n}{e_i},\frac{n}{e_i'} \right)(\HJL(z_i)+1) \right)\\
        \quad\quad & +\sum_{i=1}^r(\gcd(n,d_i)-1)\Delta_{K_i/K}+2(n-\gcd(b,n))-\begin{cases}
            0&b=0\\
            2(n-1)&b\geq 1
        \end{cases}\\
        \quad\quad= &\ \cdc(\Y_m\to...\Y_0,F)+\sum_{i=1}^r(\gcd(n,d_i)-1)\Delta_{K_i/K}+\begin{cases}
            0&b=0\\
            2-2\gcd(b,n)&b\geq1
        \end{cases} \\
         \quad\quad\geq &\ \cdc(\Y_m\to...\Y_0,F)+\begin{cases}
            0&b=0\\
            2-2\gcd(b,n)&b\geq1,
        \end{cases}
    \end{align*}
    where we use that the terms $(\gcd(n,d_i)-1)\Delta_{K_i/K}$ are nonnegative for the last inequality.
\end{proof}

Now let $\phi(X/K)$ denote the conductor exponent for the Galois
representation $\Gal (\overline{K}/K) \rightarrow $ $\Aut_{\Q_\ell}
(H^1_{\mathrm{et}}(X_{\overline{K}}, \Q_\ell))$ ($\ell \neq \cha k$),
which depends only on the generic fiber $X$ and not on the model $\mc{X}$. We can adapt Proposition \ref{prop:
  cond-disc ineq cdc version} to give a lower bound on
$(n-1)\nu_K(\disc_n(f))-\phi(X/K)$.  Recall from
Proposition~\ref{Partinconductor} that if $\mc{X}/\mc{O}_K$ is any regular model of $X$ and $N$ is the number of irreducible components of the special fiber $\mc{X}_{s}$, then
\begin{equation}\label{Eliu}
  \Art^+(\mc{X}/\mc{O}_K) = N - 1 + \phi(X/K).
\end{equation}

\begin{corollary}\label{cor: exp-disc ineq cedc version}
    Keep the notation of Proposition \ref{prop: cond-disc ineq cdc version}. Then
    \[(n-1)\nu_K(\disc_n(f))-\phi(X/K) \geq \cedc(\Y_m\to\Y_{m-1}\to...\to\Y_0,F)+\begin{cases}
            0&b=0\\
            2-2\gcd(b,n)&b\geq1.
        \end{cases}\] 
\end{corollary}
\begin{proof}
    Let $N_m'$ denote the number of irreducible components of the
    special fiber $(\X_m')_s$. By Proposition \ref{prop: cond-disc
      ineq cdc version} and \eqref{Eliu} applied
    to the regular model $\X_m'$ of $X$, it follows that 
    \begin{equation}\label{Ecdcexpression}
    \begin{aligned}
    &\quad\  (n-1)\nu_K(\disc_n(f))-\phi(X/K) \\
        &= N_m'-1+ (n-1)\nu_K(\disc_n(f))-\Art^+(\mathcal{X}_m')\\
        &\geq N_m'-1+\cdc(\Y_m\to\Y_{m-1}\to...\to\Y_0, F)+\begin{cases}
            0&b=0\\
            2-2\gcd(b,n)&b\geq1.
        \end{cases}\\
    \end{aligned}
    \end{equation}
    
   We now estimate $N_m'$. The strict transforms on $\Y_m$ of the
   special fiber $(\Y_0)_s=\P^1_k$ and of the exceptional divisor of
   each blowup $\Y_k\to\Y_{k-1}$ are distinct irreducible components
   of $(\Y_m)_s$. In turn, each irreducible component of $(\Y_m)_s$
   lies under at least one irreducible component of the special fiber
   $(\X_m)_s$ of the normalization $\X_m$ of $\Y_m$ in $K(X)$. Now let
   $z\in \Y_m$ be the intersection point of two irreducible components
   $\Gamma,\Gamma'$ of the branch divisor $D_m$ with ramification
   indices $e,e'$ respectively. By Proposition \ref{prop: structure of
     points upstairs}(iii), Definition \ref{dfn: Hirzebruch--Jung
     length of point}, and Remark~\ref{rmk: bound on HJ length of point}, there are $\gcd(\frac{n}{e},\frac{n}{e'})$
   points of $\X_m$ lying over $z$, and for each such point $x$,
   there are $\HJL(z_i)$ irreducible components of the special fiber $(\X_m')_s$ contained in the preimage of $x$ under the resolution $\X_m'\to\X_m$. It follows that
    \[N_m'\geq 1+m+\sum_{i=1}^l\gcd \left(\frac{n}{e},\frac{n}{e'} \right)\cdot\HJL(z_i). \]
    Combining this with the definition of $\cdc(\Y_m\to\Y_{m-1}\to...\to\Y_0, F)$ and writing $m = \sum_{k \colon n \nmid \nu_k} 1 + \sum_{k \colon n \mid \nu_k} 1$, we get
    \begin{align*}
        &N_m'-1+\cdc(\Y_m\to\Y_{m-1}\to...\to\Y_0, F)\\=& N_m'-1+\sum_{k:n\nmid\nu_k}[n-2\gcd(\nu_k,n)+(n-1)\mu_k(\mu_k-3)]+\sum_{k:n\mid\nu_k}[-n+(n-1)\mu_k(\mu_k-1)]\\
        &\quad+\sum_{i=1}^l \left[ \left(\frac{n}{e_i}-1 \right)+ \left(\frac{n}{e_i'}-1 \right)+\left(n-\gcd \left(\frac{n}{e_i},\frac{n}{e_i'} \right)(\HJL(z_i)+1) \right) \right]\\
        \geq& \sum_{k:n\nmid\nu_k}[1+n-2\gcd(\nu_k,n)+(n-1)\mu_k(\mu_k-3)]\\
        &\quad+\sum_{k:n\mid\nu_k} [1-n+(n-1)\mu_k(\mu_k-1)]+\sum_{i=1}^l \left[ \left(\frac{n}{e_i}-1 \right)+\left(\frac{n}{e_i'}-1 \right)+\left(n-\gcd \left(\frac{n}{e_i},\frac{n}{e_i'} \right) \right) \right]\\
        =& \cedc(\Y_m\to\Y_{m-1}\to...\to\Y_0,\F).
    \end{align*}
     Substituting this into Equation \eqref{Ecdcexpression}, we obtain the corollary.
\end{proof}

\section{Proof of the conductor-discriminant inequality via explicit analysis of multiplicity $2$ points of the branch divisor}\label{Scontribution}

In this section, we first prove the conductor exponent-discriminant inequality (\Cref{Ccedi}), which is easier to prove, and then prove the full conductor-discriminant inequality (\Cref{Tcdi}). Recall from \Cref{Lcollisions} that when the branch divisor of $\mc{X}_0 \rightarrow \mc{Y}_0$ is smooth, or equivalently, when $\nu_K(\disc_n(f)) = 0$, the normalization $\mc{X}_0$ of $\mc{Y}_0 \colonequals \P^1_{\mc{O}_K}$ in $K(X)$ is smooth, and hence the Artin conductor of $\mc{X}_0$ is $0$. When $\nu_K(\disc_n(f)) \neq 0$, the problem reduces to explicitly constructing a model $\mc{X}_m$ as the normalization of an inductively defined blowup $\mc{Y}_m$ of $\mc{Y}_0$ that partially resolves the branch locus, and such that $\X_m$ has only tame cyclic quotient singularities (\Cref{prop: reduce to mult 2}), and which has an explicit resolution $\mc{X}_m'$ with sufficiently many contractible components (\Cref{prop: non-negativity of cdc and existence of contractions}). 
\Cref{prop: reduce to mult 2} reduces the proof of \Cref{Ccedi} and \Cref{Tcdi} to the problem of explicitly understanding resolutions of multiplicity $2$ points of the branch divisor. We define explicit resolutions of these multiplicity $2$ points (\Cref{dfn: resolution of local branch divisor} and \Cref{lemma: local resolution is finite}), and then appeal to key computations of the conductor-discriminant contribution and the conductor exponent-discriminant contribution for multiplicity $2$ points from \cite[Theorems~1.6,1.7]{Part2} to deduce our results.  

\begin{notation}\label{NSection10}
    Let $\Y,\X$ be as in \Cref{notation: divisor notations}. Let $F \colonequals \divi_0(f) \subset \mc{Y}$ and let $D \colonequals (F\mod n)^{\red}$ be the branch divisor for the map $\X \rightarrow \Y$. Let $y$ be a multiplicity $2$ point of $D$, and let $F_y$ be the sum of the irreducible components of $D$ that contain $y$, with multiplicities inherited from $F$.
\end{notation}

\begin{defn}\label{dfn: resolution of local branch divisor}    
The \textit{local $n$-resolution} of $D$ at $y$ is the sequence of blowups $...\to\Y_k\to...\to\Y_0 \equalscolon \Y$ centered at reduced closed points $y_{k+1}\in\Y_k$ defined inductively as follows.
    \begin{itemize}
        \item Set $F_0'=F_0=F_y$.
        \item If the divisor $F_k'$ on $\Y_k$ contains multiple points lying over $y$ via the map $\Y_k\to\Y_0$, or if $F_k'$ contains a unique point lying over $y$ and $(F_k\tx{ mod }n)^{\red}$ has simple normal crossings at this point, the local $n$-resolution of $D$ at $y$ terminates at the stage $\Y_k\to...\to\Y_0$.
        \item Otherwise, let $y_{k+1}\in\Y_k$ be the unique point of $F_k'$ lying over $y$, let $\Y_{k+1}\to\Y_k$ be the blowup of $\Y_k$ at $y_{k+1}$, and let $F_{k+1},F_{k+1}'$ respectively denote the total and strict transforms of $F_y$ under $\Y_{k+1}\to\Y_0$, so that in particular, $F_{k+1}'$ contains at least one point lying over $y$.
\end{itemize}
\end{defn}

\begin{lemma}\label{lemma: local resolution is finite} \cite[Proposition 2.15]{Part2}
    Keep the notation of Definition \ref{dfn: resolution of local branch divisor}. The local $n$-resolution of $D$ at $y$ terminates at a finite stage $\Y_m\to...\to\Y_0$ at which $(F_m\textup{ mod }n)^{\red}$ has simple normal crossings.
\end{lemma}

The local $n$-resolution of a general multiplicity $2$ point is described explicitly in \cite[Section 3]{Part2}. We include a simple example to illustrate. 

\begin{example}\label{example: local n-resolution}
     
Suppose $\tx{char}(k)\neq 3$, and let $X$ be the smooth projective curve over $K$ with affine equation $w^3=x^3-\pi_K^2$. Let $\Y=\P^1_{\O_K}$, let $\X$ be the normalization of $\Y$ in $K(X)$ and let $D$ be the branch divisor of $\X \rightarrow \Y$. Then $D$ is the flat closure of the point of $\P^1_K$ defined by the ideal $(x^3-\pi_K^2)\subset K[x]$, the unique closed point $y$ on $D$ is defined by the ideal $(x,\pi_K)\subset \O_K[x]$, and $y$ is a singular point of $D$ of multiplicity $2$.

  An explicit computation shows that the local $3$-resolution terminates after two blow-ups $\Y_2\to\Y_1\to\Y_0=\Y$. Furthermore, if $E_1,E_2$ are the exceptional divisors of $\Y_1\to\Y_0$ and $\Y_2\to\Y_1$ respectively, and if $E_1'$ is the strict transform of $E_1$ on $\Y_2$, then
    \begin{align*}
     &F_0=F_0'=F_y=D, &(F_0\tx{ mod }3)^\tx{red}=F_0', \\
     &F_1=F_1'+2E_1, &(F_1\tx{ mod }3)^\tx{red}=F_1'+E_1, \\
     &F_2=F_2'+2E_1'+3E_2, &(F_2\tx{ mod }3)^\tx{red}=F_2'+E_1'.
    \end{align*}
where the divisors $F_1'$ and $E_1$ on $\Y_1$ intersect non-transversely, and the divisors $F_2',E_1',$ and $E_2$ on $\Y_2$ meet pairwise transversely at a single common point.

For more details of the computation above, apply \cite[Propositions 3.7, 3.11, and 3.13]{Part2} with $n=e_0=3$ and $j=1$. Note that in this case the local $3$-resolution of $y$ differs from the full embedded resolution of $y$ since the total transform $F_2$ of $D$ is not a simple normal crossings divisor; in this case the full embedded resolution of $y$ needs one additional blowup. (See, e.g., \cite[Chapter~V, Example~3.9.1]{hartshorne} for the analogous computation in the geometric setting of a surface singularity over $\mathbb{C}$, with $y$ in \cite{hartshorne} playing the role of our $\pi_K$.)
\end{example}

\begin{defn}\label{dfn: cdc of a point}
    Keep the notation of \Cref{dfn: resolution of local branch divisor} and recall \Cref{dfn: cdc of sequence}. Define the 
    \textit{conductor-discriminant contribution} $\cdc(y)$ and
    \textit{conductor exponent-discriminant contribution} $\cedc(y)$ of $y$ to be those associated to the local $n$-resolution $\Y_m\to...\to\Y_0$ of $D$ at $y$ and divisor $F_y$ on $\Y$. 
\end{defn}

The next lemma shows the quantities $\cdc(y)$ and $\cdc(\Y_m\to...\to\Y_0,\divi_0(f))$ are the same except that the crossing factor and crossing bonus in $\cdc(\Y_m\to...\to\Y_0,\divi_0(f))$ may have additional contributions  from points lying outside of the preimage of $y$ on $\Y_m$. It is an easy consequence of Definition \ref{dfn: resolution of local branch divisor} (see \cite[Lemma 2.17]{Part2} for a short argument).

\begin{lemma}\label{lemma: compatible resolutions} \cite[Lemma 2.17]{Part2}
    Suppose the local $n$-resolution of $D$ at $y$ begins with the sequence of blow-ups $\Y_k\to...\to\Y_0\equalscolon\Y$.
    \begin{enumerate}[\upshape (i)]
        \item $(F_k\textup{ mod }n)^{\red}$ is the sum of the irreducible components of $(\divi_0(f)\textup{ mod }n)^{\red}$ that intersect the preimage of $y$ on $\Y_k$.
        \item The multiplicity $2$ simple normal crossings points of $(F_k\textup{ mod }n)^{\red}$ are exactly those of $(\divi_0(f)$ $\textup{mod }n)^{\red}$ that are contained in the preimage of $y$ on $\Y_k$.
        \item If $y'\in\Y_k$ is a closed point lying over $y$ via $\Y_k\to\Y_0$, then $\mult_{y'}(\divi_0(f))\equiv \mult_{y'}(F_k^*)$ \textup{mod} $n$ and $\mult_{y'}((\divi_0(f)\textup{ mod }n)^{\red})=\mult_{y'}((F_k\textup{ mod }n)^{\red})$.
    \end{enumerate}
\end{lemma}

We also have:

\begin{prop}\label{prop: error term cancellation}\cite[Proposition 4.7]{Part2} 
    Let $D$ be as in Notation~\ref{NSection10}. Let $y$ be a multiplicity $2$ point of $D$ contained in two regular irreducible components of $D$, whose ramification indices are $e$ and $e'$. If $y$ is a simple normal crossings point of $D$, or if $e\neq e'$, then $\cedc(y)\geq \cdc(y)\geq 2(n/e)-2$.
\end{prop}

 We now explain why resolutions of multiplicity $2$ points are the bottlenecks for deducing the
 conductor-discriminant and conductor exponent-discriminant
 inequalities easily from the results thus far.  We begin by showing that we can first resolve the
 points of multiplicity $\geq 3$ in the branch divisor until there are no such points left, and we then
 construct explicit resolutions of points of multiplicity $2$ one by
 one. To that end, we define the following arithmetic schemes. For the
 remainder of the paper, for an $\mc{O}_K$-model $\Y_i$ of $\P^1_K$,
 we let $\X_i$ denote the normalization of $\Y_i$ in $K(X)$ and we let
 $D_i$ denote the branch divisor of the map $\X_i
 \rightarrow \Y_i$.
 

    \begin{lemma}\label{lem:resolveinorder}
     There is a sequence of blowups $\Y_j\to\Y_{j-1}\to...\to\Y_0\colonequals \Y$ at reduced closed points $y_k \in \Y_{k-1}$, such that
     \begin{enumerate}[\upshape (i)]
        \item for all $1\leq k\leq j$, the point $y_k$ has multiplicity $\geq 3$ in the branch divisor $D_{k-1}$, and,
        \item there are no points of multiplicity $\geq 3$ in the branch divisor $D_j$.
    \end{enumerate}
    \end{lemma}
    \begin{proof}


    Since $\O_K$ is excellent, by the theorem on the existence of embedded
    resolutions of curves on arithmetic surfaces
    \cite[Theorems 9.2.26 and 9.2.2]{Liubook}, there is a modification $\pi \colon
    \Y' \to \Y$, where $\pi$ is given as a composition 
    $$\Y' =: \Y'_r \xrightarrow{\pi_r'} \Y'_{r-1} \xrightarrow{\pi_{r-1}'} \cdots \xrightarrow{\pi_1'} 
    \Y'_0 := \Y$$ of blowups at
    reduced closed points $y'_i \in \Y'_{i-1}$, such that $\pi^*D$
    has normal crossings. In particular, the divisor $(\pi^*D)^\tx{red}$ contains no points of multiplicity $\geq 3$. Since the branch divisor $D'$ of the normalization $\X'\to\Y'$ of $\Y'$ in $K(X)$ satisfies $D'\leq (\pi^*D)^\tx{red}$, then $D'$ also contains no points of multiplicity $\geq 3$, and the sequence of blowups $\Y_r'\to...\to\Y_0'$ satisfies (ii).

    
    It is possible however that some $y_i'$
    has multiplicity $< 3$ in the appropriate branch divisor, and thus part (i) is not
    satisfied. So the idea is to do the process above, but without doing any of
    the blowups at points of multiplicity $< 3$ in the appropriate
    branch divisor.  More precisely, we prove our general lemma by induction on $r$, the minimal number of point blowups required for the pullback of $D$ to have normal crossings.  If $r = 0$, then $D$ has normal crossings and thus no points
    of multiplicity $\geq 3$, so we can take $\mc{Y}_j = \mc{Y}$.  In
    fact, regardless of $r$, if $D$ has no points of multiplicity
    $\geq 3$, we can take $\Y_j = \Y$.

    Thus we may assume that $D$ has a point $y$ of multiplicity $\geq
    3$.  Since $y$ is an isolated singularity, we may further assume that $y =
    y_0'$, that is, $\pi_1' \colon \Y'_1 \to \Y$ is the blowup at $y$. Set $\Y_1 =
    \Y_1'$. Let $\pi' \colonequals \pi_r' \circ \cdots \circ \pi_2'$. Then the pullback $(\pi')^*D_1$ of the branch divisor $D_1$ has normal crossings since $(\pi')^*D_1 \leq (\pi')^*(\pi_1^*D)$ and since  $(\pi')^*(\pi_1^*D) = (\pi' \circ \pi_1')^*D$ has normal crossings by definition of the $\pi_i'$. and since
    $D_1 \leq (\pi_1')^* D$.  By induction, there is a sequence
    of point blowups of $\Y_1$ satisfying parts (i) and (ii) of the
    lemma, so we are done.
    \end{proof}

    Let $\Y_j$ be a model satisfying the assumptions in
    \Cref{lem:resolveinorder}, and let $z_1,...,z_w$ be the points of
    multiplicity $2$ in $D_j \subset \Y_j$. For $1\leq i\leq w$, we
    inductively define schemes $\Y_{j+i}$ such that there is a
    birational morphism $\Y_{j+i} \rightarrow \Y_j$ that is an
    isomorphism away from the set $\{z_1,\ldots,z_{i-1}\}$ in $\Y_j$
    and its preimage in $\Y_{j+i}$. In particular, above each point in
    the set $\{z_{i},\ldots,z_w\}$, there will be a unique point in
    $\Y_{j+i-1}$ also of multiplicity $2$ in the corresponding branch
    divisor $D_{j+i-1}$. Define $\Y_{j+i}\to\Y_{j+i-1}$ to be the
    local $n$-resolution of the branch divisor $D_{j+i-1}$ at the
    unique point of $\Y_{j+i-1}$ lying over $z_i$.
    Let $m=j+w$.

\begin{prop}\label{prop: reduce to mult 2} Suppose $\Y=\P^1_{\O_K}$. Let $\Y_j$ be as in \Cref{lem:resolveinorder}, and define $\Y_m,\X_m,D_m$ as
above. Then the branch divisor $D_m$ of $\X_m \rightarrow
\Y_m$ has simple normal crossings, the model $\X_m$ has only tame
cyclic quotient singularities, and if $\X_m'\to\X_m$ is the minimal
regular resolution, then there is a subset $A\subseteq\{z_1,...,z_w\}$ such that
    \begin{align*}
        (n-1)\nu_K(\disc_n(f))-\Art^+(\mathcal{X}_m')&\geq\sum_{z\in A}\cdc(z),
    \end{align*}
    and
    \begin{align*}
        (n-1)\nu_K(\disc_n(f))-\phi(X/K)&\geq\sum_{z\in A}\cedc(z).
    \end{align*}
\end{prop}
\begin{proof}
    By assumption (ii) of \Cref{lem:resolveinorder}, the non-regular points of $D_j$ are exactly $z_1,...,z_w$. Since these points form an isolated set, Lemma \ref{lemma: local resolution is finite} implies $D_m$ has normal crossings, and Lemma \ref{lemma: compatible resolutions}(ii) and (iii) imply
    \begin{align*}
    \cdc(\Y_m\to...\to\Y_0,F)&=\sum_{i=1}^w\cdc(z_i)\\
    &+\sum_{k\leq j:n\nmid\nu_k}[n-2\gcd(\nu_k,n)+(n-1)\mu_k(\mu_k-3)]\\
        &+\sum_{k\leq j:n\mid\nu_k}[-n+(n-1)\mu_k(\mu_k-1)],
    \end{align*}
    where the second two sums run only over $k\leq j$, and $\mu_k,\nu_k$ denote the multiplicities of the point $y_k$ in the divisors $D_{k-1}$ and $\divi_0(f)$ on $\Y_{k-1}$ respectively.
    
    For $k\leq j$, we have $\mu_k\geq 3$ by assumption (i), and since $n\geq 2$, it follows that $-n+(n-1)\mu_k(\mu_k-1)\geq 0$ when $n \mid \nu_k$, and that $n-2\gcd(\nu_k,n)+(n-1)\mu_k(\mu_k-3)\geq0$ when $n\nmid\nu_k$. 
    It follows that
    \begin{equation}\label{EInequality1}
    \begin{aligned}
    \cdc(\Y_m\to...\to\Y_0,F)\geq\sum_{i=1}^w\cdc(z_i).
      \end{aligned}
      \end{equation}
    Similarly, we obtain
    \begin{equation}\label{EInequality2}
    \begin{aligned}
      \cedc(\Y_m\to...\to\Y_0,F)\geq\sum_{i=1}^w\cedc(z_i).
      \end{aligned}
      \end{equation}
    In the case $b=0$ or $\gcd(b,n)=1$, if we take $A=\{z_1,...,z_w\}$, the result follows by substituting \eqref{EInequality1} and \eqref{EInequality2} respectively into the inequalities of Propositions \ref{prop: cond-disc ineq cdc version} and \ref{cor: exp-disc ineq cedc version}.
    
    Now consider the case $2 \leq \gcd(b,n) < n$. Let $\Gamma$ denote the special fiber of $\Y_0$, let $\Gamma_j'$ be its strict transform under $\Y_j\to\Y_0$, and let $e \colonequals \frac{n}{\gcd{(b,n)}}$ be its ramification index. Then $2\leq e\leq\frac{n}{2}$.
    
    Suppose $j=0$, so that $\Y_j=\Y_0=\P^1_{\O_K}$ satisfies assumption (ii) of \Cref{lem:resolveinorder}. Since $X$ is connected of genus $g\geq1$, the divisor $D_0$ contains an irreducible horizontal component $\Delta$ of ramification index $e'\neq e$, which intersects $\Gamma$ at a point $y$. Since $y$ is a singular point of $D_0$, it follows from assumption (ii) of \Cref{lem:resolveinorder} that $y\in\{z_1,...,z_w\}$ and $\Delta$ is regular. Since $\Gamma$ is regular, Proposition \ref{prop: error term cancellation} implies \[\cedc(y)\geq \cdc(y)\geq 2\frac{n}{e}-2=2\gcd{(b,n)}-2\]
    If we take $A=\{z_1,...,z_w\}\setminus\{y\}$, the result follows by substituting \eqref{EInequality1} and \eqref{EInequality2} into the inequalities of \Cref{prop: cond-disc ineq cdc version} and \Cref{cor: exp-disc ineq cedc version} respectively, and pairing off the terms $\cdc(y)$ and $\cedc(y)$ of \eqref{EInequality1} and \eqref{EInequality2} with the term $2-2\gcd(b,n)$ in \Cref{prop: cond-disc ineq cdc version} and \Cref{cor: exp-disc ineq cedc version}.
    
    Suppose instead $j>0$. Then there is a point $y\in \Y_j$ at which $\Gamma_j'$ intersects the exceptional divisor $E$ of $\Y_j\to\Y_0$. Since the regular model $\Y_j$ of $\P^1_K$ has simple normal crossings by Lemma \ref{Lnormalcrossings} and since $\Gamma_j'$ and $E$ are vertical divisors on $\Y_j$, then $y$ is the intersection point of $\Gamma_j'$ with exactly one irreducible component of $E$, which is the strict transform $(E_k)'_j$ on $\Y_j$ of the exceptional divisor $E_k$ of the blowup $\Y_k\to\Y_{k-1}$ for some $1\leq k\leq j$.
    
    If $E_k$ is ramified, then $y$ is a simple normal crossings point of $D_j$ of multiplicity $2$, hence $y\in\{z_1,...,z_w\}$, and by Proposition \ref{prop: error term cancellation}, \[\cedc(y)\geq \cdc(y)\geq 2\frac{n}{e}-2=2\gcd{(b,n)}-2\]
    If we take $A=\{z_1,...,z_w\}\setminus\{y\}$, the result follows as in the previous case.

    Finally, if $E_k$ is unramified, then $n\mid\nu_k$ and $\mu_k\geq 3$, so that the blowup $\Y_k\to\Y_{k-1}$ contributes a term of \[-n+(n-1)\mu_k(\mu_k-1)\geq 5n-6> 2\frac{n}{2}-2\geq 2\gcd(b,n)-2\] to $\cdc(\Y_m\to...\to\Y_0,F)$. We can then improve inequalities \ref{EInequality1} and \ref{EInequality2} to
    \begin{equation}\label{EInequality3}
    \begin{aligned}
    \cdc(\Y_m\to...\to\Y_0,F)\geq2\gcd(b,n)-2+\sum_{i=1}^w\cdc(z_i).
    \end{aligned}
    \end{equation}
    and
    \begin{equation}\label{EInequality4}
    \begin{aligned}
      \cedc(\Y_m\to...\to\Y_0,F)\geq2\gcd(b,n)-2+\sum_{i=1}^w\cedc(z_i).
    \end{aligned}
    \end{equation}
    If we take $A=\{z_1,...,z_w\}$, the result follows by substituting \eqref{EInequality3} and \eqref{EInequality4} respectively into the inequalities of \Cref{prop: cond-disc ineq cdc version} and \Cref{cor: exp-disc ineq cedc version}.
\end{proof}

We now apply Proposition \ref{prop: reduce to mult 2} and results of
\cite{Part2} to deduce the conductor exponent-discriminant
and
conductor-discriminant inequalities for $X$ (Corollary~\ref{Ccedi} = Corollary \ref{cor:
  exponent-discriminant ineq} and Theorem~\ref{Tcdi} = Corollary \ref{cor:
  conductor-discriminant inequality}). Recall that, in light of \eqref{Eliu}, the conductor exponent-discriminant inequality
is a consequence of the conductor-discriminant inequality. Since the
conductor exponent-discriminant inequality is of independent interest
and can be deduced without the full results of \cite{Part2}, we
first extract a simpler result from \cite{Part2} and prove
the conductor exponent-discriminant inequality directly.

\begin{prop}[{\cite[Theorem 1.6]{Part2}}]\label{prop: non-negativity for cedc}

Let $y$ be as in Notation \ref{NSection10}. Then $\cedc(y)\geq0$.
\end{prop}

\begin{corollary}[= Corollary~\ref{Ccedi}]\label{cor: exponent-discriminant ineq}
The conductor exponent-discriminant inequality holds for
$f$. That is
    \[\phi(X/K)\leq (n-1)\nu_K(\disc_n(f)).\]
\end{corollary}
\begin{proof}
    This follows immediately from Propositions \ref{prop: reduce to mult 2} and \ref{prop: non-negativity for cedc}.
\end{proof}

Unfortunately, the analog of Proposition~\ref{prop: non-negativity for
  cedc} with $\cdc$ instead of $\cedc$ does not always hold, i.e., it is possible to have $\cdc(y) < 0$ for some $y$.
But whenever this happens, it turns out that, after blowing up to
obtain a normal crossings branch locus and then taking
Hirzebruch--Jung resolutions of the tame cyclic quotient singularities, one
can blow down sufficiently many $-1$-components to salvage the
conductor-discriminant inequality.  The proposition below gives the
main result we need.

Let $\Y$ be a regular model of $\proj^1_K$.
Let $\X$ be the normalization of $\Y$ in $K(X)$, let $D$ be the
branch divisor of $\X \rightarrow \Y$, and let $y \in \Y$ be a
multiplicity 2 point of $D$. Let $\Y_m\to...\to\Y_0 = \Y$ be the local
$n$-resolution of $D$ at $y$, and  let $\X_m$ be the normalization of
$\Y_m$ in $K(X)$. Let $\X'\to\X$ and $\X_m'\to\X_m$ be the minimal resolutions of $\X$ and $\X_m$ respectively.

\begin{prop}[{\cite[Theorem 1.7]{Part2}}]\label{prop: non-negativity of cdc and existence of contractions}
Assume $\cdc(y)<0$. Then the canonical morphism $\pi: \X_m' \to\X'$ contracts at least $-\cdc(y)$ irreducible components of the special fiber $(\X_{m}')_s$ contained in the preimage of $y$ on $\X_{m}'$.
\end{prop}

\begin{remark}\label{remark: rational double points}
    The condition $\cdc(y)<0$ is only satisfied under very
    restrictive conditions on the degree of $f$, the orders of $f$
    along the components of $D$ containing $y$, and the length of
    the local $n$-resolution of $D$ at $y$. In fact, by
    \cite[Theorem 1.8]{Part2}, if $\cdc(y)<0$, then for any
    point $x\in \X$ lying over $y$, the scheme
    $\Spec\hat{\O}_{\X,x}$ is a rational double point in the
    sense of \cite[Section 24]{Lip69}, and the morphism $\X'\to\X$
    includes the standard ADE resolutions of all such points. The
    proof of Proposition~\ref{prop: non-negativity of cdc and existence of contractions} is technically involved and takes the bulk of
    \cite{Part2} --- the proof of Proposition~\ref{prop:
      non-negativity for cedc} is comparatively much simpler.
\end{remark}

\begin{example}\label{example: cdc} 
    Resume the notation of Example \ref{example: local n-resolution}. Recall that
    \begin{align*}
     &F_0=F_0'=F_y=D, &(F_0\tx{ mod }3)^\tx{red}=F_0', \\
     &F_1=F_1'+2E_1, &(F_1\tx{ mod }3)^\tx{red}=F_1'+E_1, \\
     &F_2=F_2'+2E_1'+3E_2, &(F_2\tx{ mod }3)^\tx{red}=F_2'+E_1'.
    \end{align*} 
    We compute $\cdc(y):=\cdc(\Y_2\to\Y_1\to\Y_0,D)$. In the notation of \Cref{dfn: cdc of sequence}, we have
    \begin{align*}
     &\nu_1 = 2, &\mu_1=2, \\
     &\nu_2=3, &\mu_2=2,\\
     &\ord_f(F_2') = 1,  &\ord_f(E_1') = 2.
    \end{align*}
Let $z_1$ be the multiplicity $2$ point of the branch divisor on $\mc{Y}_2$ that is the intersection of the two branch components $F_2'$ and $E_1'$. By \Cref{prop: structure of points upstairs}~\ref{prop:whatappears}, it follows that the associated ramification indices for the crossing at $z_1$ are $e_1=e_1'=3$, and 
\[\tx{HJL}(z_1)=L \left(3,-\left(\frac{1}{\gcd(1,3)} \right)^{-1}\cdot\frac{2}{\gcd(2,3)} \right)=L(3,1)=1.\]

    Plugging all these invariants into \Cref{dfn: cdc of a point}, we get 
    \begin{multline*} \cdc(y)=3-2\gcd(\nu_1,3)+(3-1)\mu_1(\mu_1-3)-3+(3-1)\mu_2(\mu_2-1)+\\
    \left(\frac{3}{e_1}-1 \right)+\left(\frac{3}{e_1'}-1 \right)+\left(3-\gcd \left(\frac{3}{e_1},\frac{3}{e_1'} \right)(\tx{HJL}(z_1)+1) \right) = -1. \end{multline*}
    It remains to show if $\mc{X}_2' \rightarrow \mc{X}_2$ is the minimal resolution of the normalization $\mc{X}_2$ of $\mc{Y}_2$, then the special fiber $(\X_2')_s$ has at least one contractible component. The reader may refer to the proof of \cite[Proposition 6.1]{Part2} in the case when $n=e_0=3$ and $j=1$ for details; we simply summarize the results here. There are four irreducible components in $\mc{X}_2$ lying above the point $y$, of which three lie above the component $E_2$ in $\mc{Y}_2$, and one lies above the component $E_1'$ in $\mc{Y}_2$. Since $\HJL(z_1) = 1$, the minimal resolution of the unique singular point above $z_1$ adds another component, bringing the total number of components lying above $y$ in $\mc{X}_2'$ to five. 
     
     On the other hand, the branch divisor $D$ of the normalization $\X\to\Y$ is totally ramified, hence $\X$ contains a unique point $x$ lying above $y$, for which $\hat{\O}_{\X,x}\cong \O_K[[x,w]]/(w^3-x^3-\pi^2)$. By Case III C. of \cite[Section 24]{Lip69} and the subsequent remark beginning with "Conversely...", we see $\Spec \hat{\O}_{\X,x}$ is a rational double point admitting a $D_4$ minimal resolution. In particular, if $\X'\to\X$ is the minimal resolution of $\X$, the preimage of $y$ on $\X'$ contains four irreducible components of the special fiber $\X'_s$. It follows that the canonical morphism $\mc{X}_2'\to\mc{X}'$ contracts $1=-\cdc(y)$ irreducible component of the special fiber $(\X_2)'_s$ contained in the preimage of $y$.    
\end{example}

\begin{prop}\label{prop: difference in Artin conductors}
Keep the notation of Proposition \ref{prop: reduce to mult 2}. Let $\X_{\min}$ be the minimal proper regular $\O_K$-model of $X$. Then
\[\Art^+(\X_m')-\Art^+(\X_{\min})\geq
  -\sum_{z\in A:\cdc(z)<0}\cdc(z).\] 
\end{prop}
\begin{proof}
    Let $\pi:\X_m'\to\X_{\min}$ be the canonical contraction, and let $N_m'$ and $N_{\min}$ denote the numbers of irreducible components of the (geometric) special fibers $(\X_m')_s$ and $(\X_{\min})_s$ respectively. By Proposition \ref{prop: non-negativity of cdc and existence of contractions}, for each $z\in A$ with $\cdc(z)<0$, the map $\pi$ contracts at least $-\cdc(z)$ irreducible components of $(\X_m')_s$ contained in the preimage of $z_i$ on $\X_m'$. Since $\pi$ is a composition of blow-downs of components of the special fiber and so creates no new components in the special fiber, we obtain 
    \[N_m'-N_{\min}\geq -\sum_{z\in A:\cdc(z)<0}\cdc(z).\]
    The result then follows from \eqref{Eliu}. 
\end{proof}
    
\begin{corollary}[= Theorem~\ref{Tcdi}]\label{cor: conductor-discriminant inequality}
    The conductor-discriminant inequality holds for $f$, i.e.,
    \[\Art^+(\X_{\min})\leq (n-1)\nu_K(\disc_n(f)).\]
\end{corollary}
\begin{proof}
    This follows immediately from Propositions \ref{prop: reduce to mult 2} and \ref{prop: difference in Artin conductors}.
\end{proof}

\end{document}